\documentclass[border=10pt]{amsart}
\usepackage{amsfonts}
\usepackage{amsmath, amscd, amssymb, yhmath}
\usepackage{pgf}
\usepackage{tikz}
\usepackage{verbatim}
\usetikzlibrary{quotes,angles}
\usepackage[all,import,rotate]{xy}
\usepackage{latexsym}
\usepackage{mathrsfs}
\usepackage{url}
\xyoption{all}

\newtheorem{theorem}{Theorem}[section]

\newtheorem{lem}[theorem]{Lemma}
\newtheorem{lemma}[theorem]{Lemma}

\newtheorem{cor}[theorem]{Corollary}
\newtheorem{prop}[theorem]{Proposition}
\newtheorem*{B-principle}{B-principle}
\newtheorem*{H-principle}{H-principle}
\newtheorem{Specialthm}{Theorem}

\theoremstyle{remark}
\newtheorem{remark}[theorem]{Remark}

\theoremstyle{definition}
\newtheorem{definition}[theorem]{Definition}

\begin{document}

\title{Knot invariant with multiple skein relations}
\author{Zhiqing Yang}
\thanks{The author is supported by a grant (No. 11271058) of NSFC.}

\address{Department of Mathematics, Dalian University of Technology, China}
\email{yangzhq@dlut.edu.cn}

\date{\today}

\begin{abstract}
Given any oriented link diagram, one can construct knot invariants using skein relations. Usually such a skein relation contains three or four terms. In this paper, the author introduces several new ways to smooth a crossings, and uses a system of skein equations to construct link invariant. This invariant can also be modified by writhe to get a more powerful invariant. The modified invariant is a generalization of both the HOMFLYPT polynomial and the two-variable Kauffman polynomial. Using the diamond lemma, a simplified version of the modified invariant is given. It is easy to compute and is a generalization of the two-variable Kauffman polynomial.
\end{abstract}


\maketitle

\setcounter{tocdepth}{2}

\section{Introduction}\label{sec:1}

Polynomial invariants of links have a long history. In 1928, J.W. Alexander {\cite{Al}} discovered the famous Alexander polynomial. It has many connections with other topological invariants. More than 50 years later, in 1984 Vaughan Jones {\cite{J}}  discovered the Jones polynomial. Soon, the HOMFLYPT polynomial {\cite{HOMFLY}\cite{PT}} was found. It turns out to be a generalization of both the Alexander polynomial and the Jones polynomial. There are other polynomials, for example, the Kauffman 2-variable polynomial. All those polynomials satisfy certain skein relations, which are linear equations concerning several link diagrams. A natural questions is whether they can be further generalized. In this paper, the author presents a new approach to construct link invariant. It is a natural generalization of both the HOMFLYPT polynomial and the 2-variable Kauffman polynomial. This is a rewritten and improved  version of an earlier preprint of the author {\cite{Y}}.

For simplicity, we use the following symbols to denote link diagrams. In Fig. \ref{fig1}, letters $E,S,W,N$ mean the east, south,west and north directions as in usual maps,  $+$ means positive crossing, $-$ means negative crossing. For example, $S_+$ means the middle of the two arrows is south direction, and the crossing is of positive type. $S$ means the middle of the two arrows is south direction, and there is no crossing. Similarly, we have the local diagrams $N_+,N_-,N,$ $W_+,W_-,W,S_+,S_-,S$. The diagram $HC$ means that it is horizontal, and rotating clockwise. Similarly, $VT$ means that it is vertical, and rotating anticlockwise.

\begin{figure}[ht]
\beginpgfgraphicnamed{graphic-of-flat-world}
\begin{tikzpicture}
\draw [thick] [->]  (0pt,40pt) -- (40pt,0pt); \draw [thick] (0pt,0pt) -- (16pt,16pt); \draw [thick]  [->] (24pt,24pt) -- (40pt,40pt); \draw [thick] [thick] (15pt,-10pt) node[text width=0.4pt] {$E_+$};
\draw [thick]   [->] (100pt,0pt) -- (140pt,40pt); \draw [thick] (100pt,40pt) -- (116pt,24pt); \draw [thick]  [->] (124pt,16pt) -- (140pt,0pt); \draw [thick] (115pt,-10pt) node[text width=0.4pt] {$E_-$};
\draw [thick] [->] (200pt,0pt) .. controls (215pt,20pt) and (225pt,20pt) .. (240pt,0pt); \draw [thick] (215pt,-10pt) node[text width=0.4pt] {$E$};
\draw [thick] [->] (200pt,40pt) .. controls (215pt,20pt) and (225pt,20pt) .. (240pt,40pt);
\draw [thick]  [<-] (300pt,40pt) .. controls (320pt,25pt) and (320pt,15pt) .. (300pt,0pt); \draw [thick] (315pt,-10pt) node[text width=0.4pt] {$N$};
\draw [thick]  [<-] (340pt,40pt) .. controls (320pt,25pt) and (320pt,15pt) .. (340pt,0pt);
\draw [shift={(0,-3)}] [thick] [->]  (0pt,40pt) -- (40pt,0pt); \draw [shift={(0,-3)}] [thick]  [<-] (0pt,0pt) -- (16pt,16pt); \draw [shift={(0,-3)}] [thick]  (24pt,24pt) -- (40pt,40pt); \draw [shift={(0,-3)}] [thick] [thick] (15pt,-10pt) node[text width=0.4pt] {$S_-$};
\draw [shift={(0,-3)}] [thick]   [->] (100pt,0pt) -- (140pt,40pt); \draw [shift={(0,-3)}] [thick]  [<-]  (100pt,40pt) -- (116pt,24pt); \draw [shift={(0,-3)}] [thick] (124pt,16pt) -- (140pt,0pt); \draw [shift={(0,-3)}] [thick] (115pt,-10pt) node[text width=0.4pt] {$N_+$};
\draw [shift={(0,-3)}] [thick] [<-] (200pt,0pt) .. controls (215pt,20pt) and (225pt,20pt) .. (240pt,0pt); \draw [shift={(0,-3)}] [thick] (215pt,-10pt) node[text width=0.4pt] {$W$};
\draw [shift={(0,-3)}] [thick] [<-] (200pt,40pt) .. controls (215pt,20pt) and (225pt,20pt) .. (240pt,40pt);
\draw [shift={(0,-3)}] [thick] [<-] (300pt,40pt) .. controls (320pt,25pt) and (320pt,15pt) .. (300pt,0pt); \draw[shift={(0,-3)}]  [thick] (315pt,-10pt) node[text width=0.4pt] {$VC$};
\draw [shift={(0,-3)}] [thick]  [->] (340pt,40pt) .. controls (320pt,25pt) and (320pt,15pt) .. (340pt,0pt);
\draw [shift={(-7,-6)}] [thick] [<-] (200pt,0pt) .. controls (215pt,20pt) and (225pt,20pt) .. (240pt,0pt); \draw [shift={(-7,-6)}] [thick] (215pt,-10pt) node[text width=0.4pt] {$HC$};
\draw [shift={(-7,-6)}] [thick] [->] (200pt,40pt) .. controls (215pt,20pt) and (225pt,20pt) .. (240pt,40pt);
\draw [shift={(-7,-6)}] [thick] [->] (300pt,40pt) .. controls (320pt,25pt) and (320pt,15pt) .. (300pt,0pt); \draw[shift={(-7,-6)}]  [thick] (315pt,-10pt) node[text width=0.4pt] {$S$};
\draw [shift={(-7,-6)}] [thick]  [->] (340pt,40pt) .. controls (320pt,25pt) and (320pt,15pt) .. (340pt,0pt);
\draw [shift={(0,-6)}] [thick] [->] (200pt,0pt) .. controls (215pt,20pt) and (225pt,20pt) .. (240pt,0pt); \draw [shift={(0,-6)}] [thick] (215pt,-10pt) node[text width=0.4pt] {$HT$};
\draw [shift={(0,-6)}] [thick] [<-] (200pt,40pt) .. controls (215pt,20pt) and (225pt,20pt) .. (240pt,40pt);
\draw [shift={(0,-6)}] [thick] [->] (300pt,40pt) .. controls (320pt,25pt) and (320pt,15pt) .. (300pt,0pt); \draw[shift={(0,-6)}]  [thick] (315pt,-10pt) node[text width=0.4pt] {$VT$};
\draw [shift={(0,-6)}] [thick] [<-]  (340pt,40pt) .. controls (320pt,25pt) and (320pt,15pt) .. (340pt,0pt);
\end{tikzpicture}
\endpgfgraphicnamed
\caption{Local Diagrams with old notations.}
\label{fig1}
\end{figure}
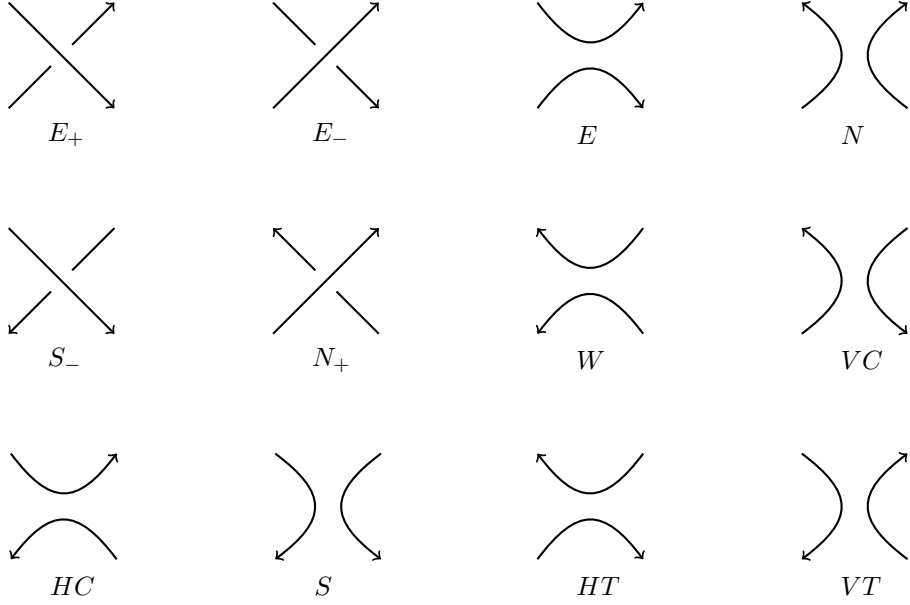

For a local crossing $E_+$ or $E_-$ of an oriented link diagram, we propose the following new skein relations. If the two arrows/arcs in the local diagram are from the same link component, then

\begin{gather}f(E_+)+bf(E_-)+c_1f(E)+c_2f(W)+c_3f(HC)+c_4f(HT)+d_1f(VC)+d_2f(VT)=0   \end{gather}

\noindent If the two arrows/arcs are from different components, then

\begin{gather}f(E_{+})+b'f(E_{-})+c_1'f(E)+c_2'f(W)+d'_1f(S)+d'_2f(N)=0.   \end{gather}

Now, let $X$ denote the quotient ring $Z[b,b',c_1,c_2,c_3,c_4,d_1,d_2,b',c_1'c_2',d_1',d_2',v_1,v_2,\cdots ]/R_1$, where $R_1=R\cup \{d_1'=d_2', (1+b+d_1+d_2)v_n+(c_1+c_2+c_3+c_4)v_{n+1}=0, $ for all $n=1,2,3, \cdots  \} $. $R$ is given in the next section. Here is our main theorem.

\begin{Specialthm}{}
For oriented link diagrams, there is a link invariant $f$ with values in $X$ and satisfies the following skein relations:

\noindent (1) If the two strands are from same link component, then

$f(E_{+})+bf(E_{-})+c_1f(E)+c_2f(W)+c_3f(HC)+c_4f(HT)+d_1f(VC)+d_2f(VT)=0.$

\noindent (2) Otherwise,
$f(E_{+})+b'f(E_{-})+c_1'f(E)+c_2'f(W)+d'_1f(S)+d'_2f(N)=0.$

The value for a trivial n-component link is $v_n$.

In general, replacing $X$ by any homomorphic image of $X$, one will get a link invariant.
\upshape
\upshape
\end{Specialthm}
\medskip

This invariant can also be modified by the writhe, like the Kauffman bracket and the Kauffman 2-variable polynomial {\cite{K}}. Let $A$ be a new variable. Let $Y$ denote the quotient ring $\\ Z[A,A^{-1},b,b',c_1,c_2,c_3,c_4,d_1,d_2,b',c_1'c_2',d_1',d_2',v_1,v_2,\cdots ]/R_2$, where $R_2=R\cup \{d_1'=d_2', AA^{-1}=1, Av_n+A^{-1}bv_n+(c_1+c_2+c_3+c_4)v_{n+1}+(d_1+d_2)v_n=0, $ for all $n=1,2,3, \cdots  \} $. Then we have the following theorem.

\begin{Specialthm}{}
There is a link invariant $F$ with values in $Y$. For oriented link diagram $D$, $F(D)=f(D)A^{-w}$ where $w$ is a the writhe of the link diagram, and $f$ satisfies the following skein relations.

\noindent (1) If the two strands are from same link component, then

$f(E_{+})+bf(E_{-})+c_1f(E)+c_2f(W)+c_3f(HC)+c_4f(HT)+d_1f(VC)+d_2f(VT)=0.$

\noindent (2) Otherwise,
$f(E_{+})+b'f(E_{-})+c_1'f(E)+c_2'f(W)+d'_1f(S)+d'_2f(N)=0.$

The value of $F$ for a trivial n-component link is $v_n$.

In general, replacing $Y$ by any homomorphic image of $Y$, one will get a link invariant.
\upshape
\upshape
\end{Specialthm}
\medskip

The coefficients of each invariant come from some commutative ring. Homomorphisms and representations of those rings define new link invariants. Some choices lead to knot polynomials. For example, if in the ring $X$ we add the following relations $c_2=c_3=c_4=d_1=d_2=c_2'=d_1'=d_2'=0$ and $b=b'$, then we get a generalized HOMFLYPT polynomial with three variables $b, c_1,c_2$. If we ask $c_1=c_1'$, then the invariant we get is equivalent to the usual HOMFLYPT polynomial by some variable change.

If we set $c_1=c_2=c_3=c_4=-z/4, c_1'=c_2'=-z/2, d_1=d_2=d_1'=d_2'=z/2$, and $b=b'=-1$, and modify it by writhe, then we can get the 2-variable Kauffman polynomial. Hence the modified invariant is a generalization of both HOMFLY polynomial and 2-variable Kauffman polynomial.

Compare with the well-known knot polynomials, there are a few differences here. (1) The skein relation has 2 cases. (2) The coefficients now are from a commutative (or non commutative) ring, and there are some nontrivial relations among them. (3) The skein relation is not local here. This means for a given oriented diagram $D$, if we use the skein relation, the diagram is not only changed locally. The orientation change is globally. To avoid contradictions, not all kinds of diagrams are allowed, and the coefficients have to satisfy certain relations. This is why we do not have a polynomial invariant. The invariant takes value in a commutative ring.

Our work was motivated by Jozef H. Przytycki and Pawel Traczyk's paper {\cite{PT}}, and V. O. Manturov's proofs in his book {\cite{M}}. Our construction and proof is a modification and improvement of their work.

\section{Full resolution commutativity}\label{sec:2}

\subsection{Orientation of diagrams}\label{sec:2.1}

For simplicity, the symbol $E_+$ ($E_-$, etc.) has two meanings in this paper. It denotes (i) the whole link diagram with the special local pattern, (ii) the value of our invariant on the diagram $E_+$. In this section, instead of writing $$f(E_+)+bf(E_-)+c_1f(E)+c_2f(W)+c_3f(HC)+c_4f(HT)+d_1f(VC)+d_2f(VT)=0,$$ we write
   $$\ \ E_++bE_-+c_1E+c_2W+c_3HC+c_4HT+d_1VC+d_2VT=0.$$ In later sections, we use $f(E_+)$ to denote the value of our invariant on the diagram $E_+$.

As mentioned before, we propose the following new skein relations. If the two arrows/arcs are from the same link component, then
$$E_{+}+bE_{-}+c_1E+c_2W+c_3HC+c_4HT+d_1VC+d_2VT=0.$$

\noindent If the two arrows/arcs are from different components, then
$$E_{+}+b'E_{-}+c_1'E+c_2'W+d'_1S+d'_2N=0.$$

Each diagram/term in the equations is canonically orientated as follows. Take the first equation for example. We can draw a disk in each of the diagrams $E_{+},E_{-},E+,W,HC,HT,VC,VT$. Outside the disks, all diagrams are all the same, inside the disks are as in Fig. \ref{fig1}. Furthermore, inside the disks, $E_{+},E_{-},E+,W,HC,HT,VC,VT$ are already oriented. Let's start with $E_+$, suppose every component of $E_+$ is oriented. Then $E_-$ and $E$, outside the disk one take the same orientation as in $E_+$. Then $E_-$ and $E$ are oriented. The link components of $W$ can be divided into two sets. One set, say A, contains components passing through the disk. Then we can extend the orientation of the disk to the whole components. For other components, we just take the same orientation as in $E_+$. The same can be done for $HC,HT,VC,VT$.

In other words, for the link components containing the arcs in the local diagram, their orientations are determined by the local diagrams. For all other components, the orientation is not changed. Since we distinguish the same/different component cases, there is no contradiction regarding to the orientation assumption.

There is no $S$ or $N$ terms in the first equation, because if the two strands are from same component, this orientation assignment will cause contradiction in orientation. Under our assumption for the orientation, the two equations are the maximal. If one add other diagrams, then there will be contradiction for orientation.

\subsection{Resolution order independence condition $f_{pq}=f_{qp}$}\label{sec:2.2}

If we want to calculate the invariant of a diagram $D$, we can start at any crossing point $p$. When we apply the formula at a crossing $p$, there are two things to check, 1. the two arcs are from same/ different component, 2. the crossing is positive or negative. We call the above information the {\bf crossing pattern} of $p$.
 The crossing pattern determines which skein equation to use and how to use it. For example, if $p$ is a negative crossing point, and the two arcs are from the same link component, then we get: $E_{-}=-b^{-1}\{E_{+}+c_1E+c_2W+c_3HC+c_4HT+d_1VC+d_2VT\}$. Hence if we have defined the invariant for $E_{+},E,\cdots $, we get the invariant for $E_-$. This is similar to the usually calculation of Jones polynomial by using skein relation. This also motivates us to define the invariant inductively. Such a procedure that write one term as a linear combination of other terms in the equation will be referred to as {\bf resolving} at $p$. We call $-b^{-1}\{E_{+}+c_1E+c_2W+c_3HC+c_4HT+d_1VC+d_2VT\}$ a {\bf linear sum}. We denote it by $f_{p}(D)$.

Given a link diagram $D$ with crossings $p_1, \cdots p_n$. Pick two crossings, say $p,q$. We can use the skein relation to resolve the diagram at a crossing $p$. The output is a linear combination of many terms. Each term involves a link diagram $D_j$. We write $f_p(D)=\sum \alpha_i f(D_j)$.  Each $D_j$ also has a crossing point corresponding to the crossing $q$. For each such diagram $D_j$, we resolve it at the point $q$. We shall get $f_q(D_j)$, a linear combination of many terms. Add the results up, we get a linear combination of linear combinations. We denote the result by $f_{pq}(D)=\sum \alpha_i f_q(D_j)$. It is the result of completely resolving at two crossing points in the order $p$ first, then $q$. Similarly, if we resolve $D$ at $q$ first, then $q$, we can get another result $f_{qp}(D)$. Now, we require that for any pair $p,q$, $f_{pq}(D)=f_{qp}(D)$. The equation $f_{pq}(D)=f_{qp}(D)$ is very important in this paper. Once this condition is satisfied, one need just a few more equations to get a link invariant. We shall discuss this condition in full detail and consider several cases.

\bigskip\noindent {\bf Easy cases.} Resolve at $p$ would not change the pattern of $q$ and vise versa.

For examples, $D$ is a disjoint union of two planar link diagrams $G_1$ and $G_2$, $p\in G_1$, and $q\in G_2$. In this case, when resolve $p$, we get diagrams $D_1, \cdots , D_k$. In all the $D_i$'s, $q$ has the same crossing pattern. In $D$, $q$ also has the same crossing pattern.

$\\$
\noindent {\bf Subcase 1.} Suppose that both $p,q$ are positive crossings, for $p$, the two arrows are from same component, for $q$, the two arrows are not from same component. When we resolve $p$, we get

$$-E_{+}=bE_{-}+c_1E+c_2W+c_3HC+c_4HT+d_1VC+d_2VT.$$

Since we are discussing two crossings here, we use $(E_-,E_+)$ to denote the first crossing $p$ is changed to $E_-$, the second crossing $q$ is $E_+$.
For the first term $bE_-$ of right hand side the above equation, we resolve at $q$ and get

$$-b(E_{-},E_+)=bb'(E_{-},E_{-})+bc_1'(E_{-},E)+bc_2'(E_{-},W)+bd'_1(E_{-},S)+bd'_2(E_{-},N).$$

Likewise, we have the following equations.

\noindent
\resizebox{14cm}{!} {$(E_+,E_+)=-\{b(E_{-},E_+)+c_1(E,E_+)+c_2(W,E_+)+c_3(HC,E_+)+c_4(HT,E_+) +d_1(VC,E_+)+d_2(VT,E_+)\}$}

\noindent
$-b(E_{-},E_+)=bb'(E_{-},E_{-})+bc_1'(E_{-},E)+bc_2'(E_{-},W)+bd'_1(E_{-},S)+bd'_2(E_{-},N)\\$
$-c_1(E,E_+)=c_1b'(E,E_{-})+c_1c_1'(E,E)+c_1c_2'(E,W)+c_1d'_1(E,S)+c_1d'_2(E,N)\\$
$-c_2(W,E_+)=c_2b'(W,E_{-})+c_2c_1'(W,E)+c_2c_2'(W,W)+c_2d'_1(W,S)+c_2d'_2(W,N)\\$
\resizebox{14cm}{!} {$-c_3(HC,E_+)=c_3b'(HC,E_{-})+c_3c_1'(HC,E)+c_3c_2'(HC,W)+c_3d'_1(HC,S)+c_3d'_2(HC,N) $}

\noindent
$-c_4(HT,E_+)=c_4b'(HT,E_{-})+c_4c_1'(HT,E)+c_4c_2'(HT,W)+c_4d'_1(HT,S)+c_4d'_2(HT,N)\\$
\resizebox{14cm}{!} {$-d_1(VC,E_+)=d_1b'(VC,E_{-})+d_1c_1'(VC,E)+d_1c_2'(VC,W)+d_1d'_1(VC,S)+d_1d'_2(VC,N) $ }

\noindent
$-d_2(VT,E_+)=d_2b'(VT,E_{-})+d_2c_1'(VT,E)+d_2c_2'(VT,W)+d_2d'_1(VT,S)+d_2d'_2(VT,N)\\$

We can build a table for this result. We put the crossing type of the first crossing in the first column, the crossing type of the second crossing in the first row.

\begin{table}[ht]
\caption{Trivial case, resolving $p$ first.}\label{tab:1}
\begin{center}
\begin{tabular}[pos]{|l|l|l|l|l|l|l|l|l|l|l| }
\hline  p $\backslash $ q & $E_-$ & $E$ & $W$ & $S$ & N     \\
\hline $E_-$ & $bb'$ & $bc_1'$ & $bc_2'$  & $bd_1'$  &  $bd_2'$    \\
\hline $E$ & $c_1b'$ & $c_1c_1'$ & $c_1c_2'$  & $c_1d_1'$  &  $c_1d_2'$    \\
\hline $W$ & $c_2b'$ & $c_2c_1'$ & $c_2c_2'$  & $c_2d_1'$  &  $c_2d_2'$   \\
\hline $HC$ & $c_3b'$ & $c_3c_1'$ & $c_3c_2'$  & $c_3d_1'$  &  $c_3d_2'$   \\
\hline $HT$ & $c_4b'$ & $c_4c_1'$ & $c_4c_2'$  & $c_4d_1'$  &  $c_4d_2'$      \\
\hline $VC$ & $d_1b'$ & $d_1c_1'$ & $d_1c_2'$  & $d_1d_1'$  &  $d_1d_2'$    \\
\hline $VT$ & $d_2b'$ & $d_2c_1'$ & $d_2c_2'$  & $d_2d_1'$  &  $d_2d_2'$    \\
\hline
\end{tabular}
\end{center}
\end{table}

Other other hand, if we resolve at $q$ first, we shall get another table.

\begin{table}[ht]
\caption{Trivial case, resolving $q$ first.}\label{tab:2}
\begin{center}
\begin{tabular}[pos]{|l|l|l|l|l|l|l|l|l|l|l| }
\hline  p $\backslash $ q & $E_-$ & $E$ & $W$ & $S$ & N     \\
\hline $E_-$ & $b'b$ & $c_1'b$ & $c_2'b$  & $d_1'b$  &  $d_2'b$    \\
\hline $E$ & $b'c_1$ & $c_1'c_1$ & $c_2'c_1$  & $d_1'c_1$  &  $d_2'c_1$    \\
\hline $W$ & $b'c_2$ & $c_1'c_2$ & $c_2'c_2$  & $d_1'c_2$  &  $d_2'c_2$   \\
\hline $HC$ & $b'c_3$ & $c_1'c_3$ & $c_2'c_3$  & $d_1'c_3$  &  $d_2'c_3$   \\
\hline $HT$ & $b'c_4$ & $c_1'c_4$ & $c_2'c_4$  & $d_1'c_4$  &  $d_2'c_4$      \\
\hline $VC$ & $b'd_1$ & $c_1'd_1$ & $c_2'd_1$  & $d_1'd_1$  &  $d_2'd_1$    \\
\hline $VT$ & $b'd_2$ & $c_1'd_2$ & $c_2'd_2$  & $d_1'd_2$  &  $d_2'd_2$    \\
\hline
\end{tabular}
\end{center}
\end{table}

Compare the results, the easiest way to make them equal is to ask the coefficients equal each other. Therefor, we ask any element from the set $\{b,c_1,c_2,c_3,c_4,d_1,d_2\}$ commutes with any element from the set $\{b',c_1',c_2',d_1',d_2'\}$.

\bigskip Let $\overline{b}=b^{-1}$, $\overline{c}_i=b^{-1}c_i$ for $i=1,2,3,4$,  $\overline{d}_i=b^{-1}d_i$ for $i=1,2$.
$\overline{b}'={b'}^{-1}$, $\overline{c}_i'={b'}^{-1}c_i'$ for $i=1,2,3,4$,  $\overline{d}_i'={b'}^{-1}d_i'$ for $i=1,2$. For any pair of such symbols $x$ and $\overline{x}$, we call them the {\bf conjugates} of each other. This has an obvious benefit as follows. In a skein relation, for example $E_{+}+bE_{-}+c_1E+c_2W+c_3HC+c_4HT+d_1VC+d_2VT=0$, we can get $E_{+}=-\{bE_{-}+c_1E+c_2W+c_3HC+c_4HT+d_1VC+d_2VT\}$ and $E_{-}=-\{\overline{b}E_{+}+\overline{c}_1E+\overline{c}_2W+\overline{c}_3HC+\overline{c}_4HT+\overline{d}_1VC+\overline{d}_2VT\}.$ This means if we change $E_+$ to $E_-$ (or $E_-$ to $E_+$), we can simply replace each $x$ to $\overline{x}$. The symmetry between them will greatly simplify our discussion later.

When we list all the subcases, we get the conclusion that any two elements from $$\{b,c_1,c_2,c_3,c_4,d_1,d_2, b',c_1',c_2',d_1',d_2'\}\cup \{\overline{b},\overline{c}_1,\overline{c}_2,\overline{c}_3,\overline{c}_4,\overline{d}_1,\overline{d}_2, \overline{b}',\overline{c}_1',\overline{c}_2',\overline{d}_1',\overline{d}_2'\}$$ are mutually commutative.

$\\$
\noindent {\bf Convention}: For convenience, later on in the second table, we exchange the order of the elements of all the terms. For example, $cd$ is changed to $dc$. So for an entry $xy$, $x$ always comes from resolving the first crossing point, $y$ always comes from resolving the second crossing point.

\bigskip\noindent {\bf Nontrivial cases.}
Now we are going to discuss the nontrivial cases. For simplicity, we use $A,B$ to denote the end of the first crossing $p$, and $C,D$ to denote the end of the second crossing $q$. We also use them to denote the oriented strands. For example, $(ACB,D)$ means that the three arcs $A,C,B$ are from a same link component, and their order is $A \to C \to B$ along the link orientation. $D$ is in another component. $(AC,B,D)$ means that the arcs $A,C$ are from a same link component, and their order is $A \to C$ along the link orientation. $B$ is in the second component, $D$ is in the third component.

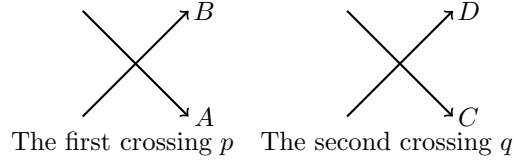
\begin{figure}[ht]
\begin{tikzpicture}
\draw [thick] [->]  (0pt,40pt) -- (40pt,0pt);
\draw [thick] [thick] (46pt,0pt) node   { $A$};
\draw [thick] [->] (0pt,0pt) -- (40pt,40pt);
\draw [thick] [thick] (46pt,40pt) node   { $B$};
\draw [thick] [thick] (15pt,-10pt) node   {The first crossing $p$};
\draw [thick]   [->] (100pt,0pt) -- (140pt,40pt);
\draw [thick] [thick] (146pt,0pt) node   { $C$};
\draw [thick]  [->] (100pt,40pt) -- (140pt,0pt);
\draw [thick] [thick] (146pt,40pt) node   { $D$};
\draw [thick] (115pt,-10pt) node  {The second crossing $q$};
\end{tikzpicture}
\caption{The label of two crossings.}\label{f4}
\end{figure}

Since there is a symmetry of positive/negative crossing both in the skein relation and the diagrams, we discuss only positive crossings cases. We shall tell how to deal with the other cases later.

For the nontrivial cases, there must be one link component passing through both the crossings $p$ and $q$. There are four arcs in the two disks containing the two crossings. They can belong to 1,2,3 link components. If there are three components, then all the possible cases are $(AC,B,D),(AD,B,C),(BC,A,D),(BD,A,C)$.

$\\$
\noindent {\bf Case 1: $(AC,B,D)$} If we resolve the 1st crossing point $p$ first, we shall get the followings.

$(E_+,E_+)=-\{(b'E_{-}+c_1'E+d_1'S,E_+)+(c_2'W+d_2'N,N_-)\}\\$
$-(b'E_{-},E_+)=b'\{(E_{-},b'E_{-}+c_1'E+d_1'S)+(N_+,c_2'W+d_2'N)\}\\$
$-(c'_1E,E_+)=c_1'\{(E,b'E_{-}+c_1'E+d_1'S)+(W,c_2'W+d_2'N)\}\\$
$-(d_1'S,E_+)=d_1'\{(S,b'E_{-}+c_1'E+d_1'S)+(N,c_2'W+d_2'N)\}\\$
$-(c_2'W,N_-)=c_2'\{(W,\overline{b}'N_{+}+\overline{c}_1'N+\overline{d}_2'W)+(E,\overline{c}_2'S+\overline{d}_1'E)\}\\$
$-(d_2'N,N_-)=d_2'\{(N,\overline{b}'N_{+}+\overline{c}_1'N+\overline{d}_2'W)+(S,\overline{c}_2'S+\overline{d}_1'E)\}$

\begin{table}[ht]
\caption{Case (AC,B,D), resolving $p$ first.}\label{tab:1a}
\begin{center}
\begin{tabular}[pos]{|l|l|l|l|l|l|l|l|l|l|l|l|}
\hline  p $\backslash $ q & E & $E_-$ & $N_+$ & $N$ & S & W   \\
\hline $E$ & $c_2'\overline{d}_1',c_1'c_1'$ & $c_1'b'$ &   &  &  $c_2'\overline{c_2}',c_1'd_1'$ &      \\
\hline $E_-$ & $b'c_1'$ & $b'b'$ &   &  &  $b'd_1'$ &     \\
\hline $N_+$ &   &  &   & $b'd_2'$ &    & $b'c_2'$     \\
\hline $N$ &   &  & $d_2'\overline{b}'$ & $d_1'd_2',d_2'\overline{c_1}'$ &    & $d_1'2',d_2'\overline{d}_2'$   \\
\hline $S$ & $d_2'\overline{d}_1',d_1'c_1'$ & $d_1'b'$ &   &  & $d_2'\overline{c_2}',d_1'd_1'$ &      \\
\hline $W$ &   &   & $c_2'\overline{b}'$ & $2'\overline{c_1}',c_1'd_2'$ &    & $c_2'\overline{d}_2',c_1'c_2'$    \\
\hline
\end{tabular}
\end{center}
\end{table}

Otherwise, we shall get the followings.

$(E_+,E_+)=-\{(E_+,b'E_{-}+c_1'E+d_1'S)+(N_-,c_2'W+d_2'N)\}\\$
$-b'(E_+,E_{-})=b'\{(b'E_{-}+c_1'E+d_1'S,E_{-})+(c_2'W+d_2'N,N_+)\}\\$
$-c'_1(E_+,E)=c_1'\{(b'E_{-}+c_1'E+d_1'S,E)+(c_2'W+d_2'N,W)\}\\$
$-d_1'(E_+,S)=d_1'\{(b'E_{-}+c_1'E+d_1'S,S)+(c_2'W+d_2'N,N)\}\\$
$-c_2'(N_-,W)=c_2'\{(\overline{b}'N_{+}+\overline{c}_1'N+\overline{d}_2'W,W)+(\overline{c}_2'S+\overline{d}_1'E,E)\}\\$
$-d_2'(N_-,N)=d_2'\{(\overline{b}'N_{+}+\overline{c}_1'N+\overline{d}_2'W,N)+(\overline{c}_2'S+\overline{d}_1'E,S)\}$

\begin{table}[ht]
\caption{Case (AC,B,D), resolving $q$ first.}\label{tab:1b}
\begin{center}
\begin{tabular}[pos]{|l|l|l|l|l|l|l|l|l|l|l|l|}
\hline  p $\backslash $ q & E & $E_-$ & $N_+$ & $N$ & S & W   \\
\hline $E$ & $\overline{d}_1'c_2',c_1'c_1'$ & $c_1'b'$ &   &  &  $\overline{d}'_1d_2',c_1'd_1'$ &      \\
\hline $E_-$ & $b'c_1'$ & $b'b'$ &   &  &  $b'd_1'$ &     \\
\hline $N_+$ &   &  &   & $\overline{b}'d_2'$ &    & $\overline{b}'c_2'$     \\
\hline $N$ &   &  & $d_2'b'$ & $d_2'd_1',\overline{c_1}'d_2'$ &    & $\overline{c_1}'c_2',d_2'c_1'$   \\
\hline $S$ & $\overline{c_2}'c_2',d_1'c_1'$ & $d_1'b'$ &   &  & $\overline{c_2}'d_2',d_1'd_1'$ &      \\
\hline $W$ &   &   & $c_2'b'$ & $\overline{d}_2'd_2',c_2'd_1'$ &    & $\overline{d}_2'c_2',c_2'c_1'$    \\
\hline
\end{tabular}
\end{center}
\end{table}

Recall that in the second matrix we write the products in a new form. For example,  $(N,N_+)$ has coefficient $b'd_2'$, but we write  $d_2'b'$ in the table. We exchange the order of every product in this matrix so that the first symbol, for example the $d_2'$ here, is always from resolving the first crossing point $p$.

\noindent The relations here are:   $c_2'\overline{d}_1'=\overline{d}_1'c_2'$, $c_2'\overline{c_2}'=\overline{d}'_1d_2'$, $b'd_2'=\overline{b}'d_2'$, $b'c_2'=\overline{b}'c_2'$, $d_2'\overline{b}'=d_2'b'$, $d_1'd_2'+d_2'\overline{c_1}'=d_2'd_1'+\overline{c_1}'d_2'$, $d_1'c_2'+d_2'\overline{d}_2'=\overline{c_1}'c_2'+d_2'c_1'$, $d_2'\overline{d}_1'=\overline{c_2}'c_2'$, $d_2'\overline{c_2}'=\overline{c_2}'d_2'$, $c_2'\overline{b}'=c_2'b'$, $c_2'\overline{c_1}'+c_1'd_2'=\overline{d}_2'd_2'+c_2'd_1'$, $c_2'\overline{d}_2'+c_1'c_2'=\overline{d}_2'c_2'+c_2'c_1'$.

\bigskip
If the two crossings are both negative crossings, then we change all the coefficients $x$ to $\overline{x}$. For example, $(E_+,E_+)=-\{(b'E_{-}+c_1'E+d_1'S,E_+)+(c_2'W+d_2'N,N_-)\}$ become $(E_-,E_-)=-\{(\overline{b}'E_{-}+\overline{c_1}'E+\overline{d_1}'S,E_-)+(\overline{c_2}'W+\overline{d_2}'N,N_+)\}.$ The relations become their conjugates. For example, $c_2'\overline{d}_1'=\overline{d}_1'c_2'$ become $\overline{c_2}'d_1'=d_1'\overline{c_2}'$.

If the first crossing $p$ is a negative crossing, $q$ is positive, then we change the first coefficients $x$ to $\overline{x}$. For example, $(E_+,E_+)=-\{(b'E_{-}+c_1'E+d_1'S,E_+)+(c_2'W+d_2'N,N_-)\}$ become $(E_-,E_+)=-\{(\overline{b}'E_{-}+\overline{c_1}'E+\overline{d_1}'S,E_+)+(\overline{c_2}'W+\overline{d_2}'N,N_-)\}.$ In the relations, we change the first variables to their conjugates. For example, $c_2'\overline{d}_1'=\overline{d}_1'c_2'$ become  $\overline{c_2}'\overline{d}_1'=d_1'c_2'$. Likewise, if the first crossing $p$ is a positive crossing, $q$ is negative, we get $c_2'd_1'=\overline{d}_1'\overline{c_2}'$. In short, if we have a relation $xy=cd$, we will add $\overline{x}y=\overline{c}d$, $x\overline{y}=c\overline{d}$ and $\overline{xy}=\overline{cd}$. We will refer to this as complete the relation by the $\overline{\ }$ operation.

To handle the $(BC,A,D)$ case, let's first introduce another conjugation induced by taking mirror image. Taking the mirror image of each term of our skein relation, $E_+,E_-,E,W,HC,HT,VC,\\ $ $VT,S,N$ are changed to  $E_-,E_+,E,W,HC,HT,VC,VT,S,N$ (see Fig. \ref{fig3}). Let $\widehat{c_3}=c_4,\widehat{c_4}=c_3,\widehat{d_1}=d_2,\widehat{d_2}=d_1,\widehat{d_1}'=d_2',\widehat{d_2}'=d_1'$. For other $x$, $\widehat{x}=x$. For the link $(BC,A,D)$, suppose the crossing $p$ is negative, $q$ is positive. Then we can change the disk at $p$ to its mirror image, and add virtual crossings. Then the new link is the case $(AC,B,D)$. In the new link, both crossings are positive. Although the new link  $(AC,B,D)$ contain virtual crossings, all the calculations we made before are still valid. There is a one to one correspondence between the results of complete resolving $(BC,A,D)$ and $(AC,B,D)$ at $p,q$. From the results of $(BC,A,D)$ to $(AC,B,D)$, the mirror takes $E_-$ to $E_+$ , $VC$ to $VT$ and so on. Since $E_-$ is mapped to $E_+$, we have to map $x$ to $\overline{x}$. Since $VC$ is mapped to $VT$, we have to map $c_3$ to $c_4=\widehat{c_3}$ and so on. Because the mirror is only placed near the first crossing $p$, in a relation $xy=cd$ we only change the first variables to get $\overline{\widehat{x}}y=\overline{\widehat{c}}d$. Therefor, if we have a relation $xy=cd$ from $(AC,B,D)$, we will add $\overline{\widehat{x}}y=\overline{\widehat{c}}d$ for $(BC,A,D)$. Since we also have $\overline{x}y=\overline{c}d$, we can say that if we have a relation $xy=cd$, we will add $\widehat{x}y=\widehat{c}d$.
Sinilarly, for $(AD,B,C),(BD,A,C)$, if we have a relation $xy=cd$, we will add $x\widehat{y}=c\widehat{d}$ and $\widehat{x}\widehat{y}=\widehat{c}\widehat{d}$. We will refer to this as complete the relation by the $\widehat{\ }$ operation.

In short, if there are three components, then all the possible cases are $(AC,B,D),(AD,B,C),\\$ $(BC,A,D),(BD,A,C)$, but we only need to calculate the case $(AC,B,D)$ and suppose that all crossings are positive.

\begin{figure}[ht]
\beginpgfgraphicnamed{graphic-of-flat-world}
\begin{tikzpicture}[scale=.7]
\draw [shift={(0,-3)}] [thick] [<-] (200pt,0pt) .. controls (215pt,20pt) and (225pt,20pt) .. (240pt,0pt); \draw [shift={(0,-3)}] [thick] (215pt,-10pt) node[text width=0.4pt] {$W$};
\draw [shift={(0,-3)}] [thick] [<-] (200pt,40pt) .. controls (215pt,20pt) and (225pt,20pt) .. (240pt,40pt);
\draw [shift={(-7,-3)}] [thick] [->] (200pt,0pt) .. controls (215pt,20pt) and (225pt,20pt) .. (240pt,0pt); \draw [shift={(-7,-3)}] [thick] (215pt,-10pt) node[text width=0.4pt] {$HT$};
\draw [shift={(-7,-3)}] [thick] [<-] (200pt,40pt) .. controls (215pt,20pt) and (225pt,20pt) .. (240pt,40pt);
\draw [shift={(-7,-3)}] [thick] [<-] (300pt,40pt) .. controls (320pt,25pt) and (320pt,15pt) .. (300pt,0pt); \draw[shift={(-7,-3)}]  [thick] (315pt,-10pt) node[text width=0.4pt] {$N$};
\draw [shift={(-7,-3)}] [thick]  [<-] (340pt,40pt) .. controls (320pt,25pt) and (320pt,15pt) .. (340pt,0pt);
\draw [shift={(0,-6)}] [thick] [<-] (200pt,0pt) .. controls (215pt,20pt) and (225pt,20pt) .. (240pt,0pt); \draw [shift={(0,-6)}] [thick] (215pt,-10pt) node[text width=0.4pt] {$W$};
\draw [shift={(0,-6)}] [thick] [<-] (200pt,40pt) .. controls (215pt,20pt) and (225pt,20pt) .. (240pt,40pt);

\draw [shift={(0,-3)}] [thick] [<-] (300pt,40pt) .. controls (320pt,25pt) and (320pt,15pt) .. (300pt,0pt); \draw[shift={(0,-3)}]  [thick] (315pt,-10pt) node[text width=0.4pt] {$VC$};
\draw [shift={(0,-3)}] [thick]  [->] (340pt,40pt) .. controls (320pt,25pt) and (320pt,15pt) .. (340pt,0pt);
\draw [shift={(-7,-6)}] [thick] [<-] (200pt,0pt) .. controls (215pt,20pt) and (225pt,20pt) .. (240pt,0pt); \draw [shift={(-7,-6)}] [thick] (215pt,-10pt) node[text width=0.4pt] {$HC$};
\draw [shift={(-7,-6)}] [thick] [->] (200pt,40pt) .. controls (215pt,20pt) and (225pt,20pt) .. (240pt,40pt);
\draw [shift={(-7,-6)}] [thick] [->] (300pt,40pt) .. controls (320pt,25pt) and (320pt,15pt) .. (300pt,0pt); \draw[shift={(-7,-6)}]  [thick] (315pt,-10pt) node[text width=0.4pt] {$S$};
\draw [shift={(-7,-6)}] [thick]  [->] (340pt,40pt) .. controls (320pt,25pt) and (320pt,15pt) .. (340pt,0pt);
\draw [shift={(0,-6)}] [thick] [->] (300pt,40pt) .. controls (320pt,25pt) and (320pt,15pt) .. (300pt,0pt); \draw[shift={(0,-6)}]  [thick] (315pt,-10pt) node[text width=0.4pt] {$VT$};
\draw [shift={(0,-6)}] [thick] [<-]  (340pt,40pt) .. controls (320pt,25pt) and (320pt,15pt) .. (340pt,0pt);
\draw [thick]  (0,-4) -- (12,-4); \node at (13,-4) {mirror};
\end{tikzpicture}
\endpgfgraphicnamed
\caption{Mirror symmetry.}
\label{fig3}
\end{figure}
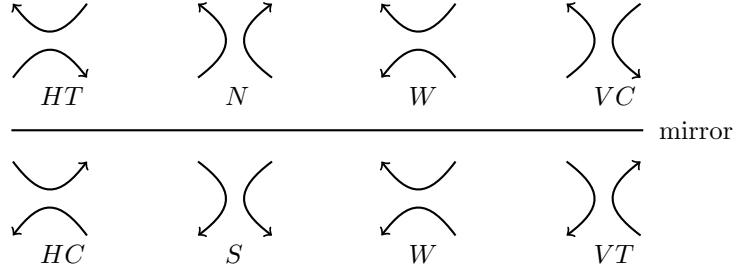

\bigskip

To get all the equations $f_{pq}=f_{qp}$, we shall list all the possible cases that how the two strands of $p$ is connected to the two strands of $q$. Up to the positive/negative crossing type symmetry,  and mirror symmetry, there are only few nontrivial cases. If there are only two components pass the two disks, up to symmetry, we have $(AC,BD)$, $(ABC,D)$. If there is only one components pass the two disks, up to symmetry, we have  $(ACBD)$ or $(ACDB)$.

So, we have the following five cases. 1. $(AC,B,D)$, 2. $(AC,BD)$, 3. $(ABC,D)$, 4. $(ACBD)$, 5. $(ACDB)$.

$\\$
\noindent {\bf Case 2, $(AC,BD)$} Resolving $p$ first, we shall get the following equations.

$(E_+,E_+)=-\{(b'E_{-}+c_1'E,E_+)+(c_2'W,W_+)+(d_2'N,N_-)+(d_1'S,S_-)\}\\$
$-(b'E_{-},E_+)=b'\{(E_{-},b'E_{-}+c_1'E)+(W_-,c_2'W)+(N_+,d_2'N)+(S_+,d_1'S)\}\\$
\noindent \resizebox{12.5cm}{!} {$-(c'_1E,E_+)=c_1'\{(E,bE_{-}+c_1E)+(W,c_2W)+(HT,c_3HC)+(HC,c_4HT)+(HC,d_1VC)+(HT,d_2VT)\}\\$}

\noindent \resizebox{12.5cm}{!} {$-(c_2'W,W_+)=c_2'\{(W,bW_{-}+c_1W)+(E,c_2E)+(HT,c_3HC)+(HC,c_4HT)+(HC,d_1VC)+(HT,d_2VT)\}\\$}

\noindent \resizebox{12.5cm}{!}
{$-(d_2'N,N_-)=d_2'\{(N,\overline{b}N_{+}+\overline{c}_1N)+(S,\overline{c}_2S)+(VT,\overline{c}_3VC)+(VC,\overline{c}_4VT)+(VC,\overline{d}_1HC)+(VT,\overline{d}_2HT)\}\\$}

\noindent \resizebox{12.5cm}{!} {$-(d_1'S,S_-)=d_1'\{(S,\overline{b}S_{+}+\overline{c}_1S)+(N,\overline{c}_2N)+(VT,\overline{c}_3VC)+(VC,\overline{c}_4VT)+(VC,\overline{d}_1HC)+(VT,\overline{d}_2HT)\}$}

\begin{table}[ht]
\caption{Case  $(AC,BD)$, resolving $p$ first.}\label{tab:2a}
\begin{center}
\resizebox{12.5cm}{!}{\begin{tabular}[pos]{|l|l|l|l|l|l|l|l|l|l|l|l|l|l|}
\hline  p $\backslash $ q & E & $E_-$ & W & $W_-$ & $N$ &$N_+$ & S & $S_+$ & HC & HT & VC & VT   \\
\hline $E$ & $c_2'c_2,c_1'c_1$ & $c_1'b$ &   &  &  &   & &&&&&  \\
\hline $E_-$ & $b'c_1'$ & $b'b'$ &   &  &  &   & &&&&&   \\
\hline $W$ &   &   & $c_2'c_1,c_1'c_2$ & $c_2'b$ &  &   & &&&&&   \\
\hline $W_-$ &   &   & $b'c_2'$ &  &  &   & &&&&&   \\
\hline $N$ &   &   &  &  & $d_1'\overline{c_2},d_2'\overline{c_1}$ & $d_2'\overline{b}$ & &&&&&   \\
\hline $N_+$ &   &   &  &  & $b'd_2'$ & & &&&&&   \\
\hline $S$ &   &   &  &  &   & & $d_1'\overline{c_1},d_2'\overline{c_2}$ & $d_1'\overline{b}$ &&&&   \\
\hline $S_+$ &   &   &  &  &   & & $b'd_1'$ &&&&&   \\
\hline $HC$ &   &   &  &  &   & &   & & & $c_1'c_4,c_2'c_4$ & $c_1'd_1,c_2'd_1$ &   \\
\hline $HT$ &   &   &  &  &   & &   & & $c_2'c_3,c_1'c_3$ & & & $c_2'd_2,c_1'd_2$ \\
\hline $VC$ &   &   &  &  &   & &   & & $d_2'\overline{d}_1,d_1'\overline{d}_1$ & & & $d_2'\overline{c_4},d_1'\overline{c_4}$  \\
\hline $VT$ &   &   &  &  &   & &   & & & $d_2'\overline{d}_2,d_1'\overline{d}_2$ & $d_1'\overline{c_3},d_2'\overline{c_3}$ &   \\
\hline
\end{tabular}}
\end{center}
\end{table}

Resolving $q$ first, we shall get the following equations.

$(E_+,E_+)=-\{(E_+,b'E_{-}+c_1'E)+(W_+,c_2'W)+(N_-,d_2'N)+(S_-,d_1'S)\}\\$
$-b'(E_+,E_{-})=b'\{(b'E_{-}+c_1'E,E_{-})+(c_2'W,W_-)+(d_2'N,N_+)+(d_1'S,S)\}\\$
\noindent \resizebox{12.5cm}{!} {$-c'_1(E_+,E)=c_1'\{(bE_{-}+c_1E,E)+(c_2W,W)+(c_3HC,HT)+(c_4HT,HC)+(d_1VC,HC)+(d_2VT,HT)\}\\$}

\noindent \resizebox{12.5cm}{!} {$-c_2'(W_+,W)=c_2'\{(bW_{-}+c_1W,W)+(c_2E,E)+(c_3HC,HT)+(c_4HT,HC)+(d_1VC,HC)+(d_2VT,HT)\}\\$}

\noindent \resizebox{12.5cm}{!} {$-d_2'(N_-,N)=d_2'\{(\overline{b}N_{+}+\overline{c}_1N,N)+(\overline{c}_2S,S)+(\overline{c}_3VC,VT)+(\overline{c}_4VT,VC)+(\overline{d}_1HC,VC)+(\overline{d}_2HT,VT)\}\\$}

\noindent \resizebox{12.5cm}{!} {$-d_1'(S_-,S)=d_1'\{(\overline{b}S_{+}+\overline{c}_1S,S)+(\overline{c}_2N,N)+(\overline{c}_3VC,VT)+(\overline{c}_4VT,VC)+(\overline{d}_1HC,VC)+(\overline{d}_2HT,VT)\}$}
\begin{table}[ht]
\caption{Case  $(AC,BD)$, resolving $q$ first.}\label{tab:2b}
\begin{center}
\resizebox{12.5cm}{!}{\begin{tabular}[pos]{|l|l|l|l|l|l|l|l|l|l|l|l|l|l|}
\hline  p $\backslash $ q & E & $E_-$ & W & $W_-$ & $N$ &$N_+$ & S & $S_+$ & HC & HT & VC & VT   \\
\hline $E$ & $c_2c_2',c_1c_1'$ & $c_1'b'$ &   &  &  &   & &&&&&  \\
\hline $E_-$ & $bc_1'$ & $b'b'$ &   &  &  &   & &&&&&   \\
\hline $W$ &   &   & $c_1c_2',c_2c_1'$ & $c_2'b'$ &  &   & &&&&&   \\
\hline $W_-$ &   &   & $bc_2'$ &  &  &   & &&&&&   \\
\hline $N$ &   &   &  &  & $\overline{c_2}d_1',\overline{c_1}d_2'$ & $d_2'b'$ & &&&&&   \\
\hline $N_+$ &   &   &  &  & $\overline{b}d_2'$ & & &&&&&   \\
\hline $S$ &   &   &  &  &   & & $\overline{c_1}d_1',\overline{c_2}d_2'$ & $d_1'b'$ &&&&   \\
\hline $S_+$ &   &   &  &  &   & & $\overline{b}d_1'$ &&&&&   \\
\hline $HC$ &   &   &  &  &   & &   & & & $c_3c_1',c_3c_2'$ & $\overline{d}_1d_2',\overline{d}_1d_1'$ &   \\
\hline $HT$ &   &   &  &  &   & &   & & $c_4c_1',c_4c_2'$ & & & $\overline{d}_2d_2',\overline{d}_2d_1'$ \\
\hline $VC$ &   &   &  &  &   & &   & & $d_1c_1',d_1c_2'$ & & & $\overline{c_3}d_1',\overline{c_3}d_2'$  \\
\hline $VT$ &   &   &  &  &   & &   & & & $d_2c_1',d_2c_2'$ & $\overline{c_4}d_1',\overline{c_4}d_2'$ &   \\
\hline
\end{tabular}}
\end{center}
\end{table}

\noindent The relations here are:  $b'c_1'=bc_1'$, $c_2'c_1+c_1'c_2=c_1c_2'+c_2c_1'$, $c_2'c_2+c_1'c_1=c_2c_2'+c_1c_1'$,$b'c_2'=bc_2'$, $d_1'\overline{c_2}+d_2'\overline{c_1}=\overline{c_2}d_1'+\overline{c_1}d_2'$, $d_2'\overline{b}=d_2'b'$, $b'd_2'=\overline{b}d_2'$, $d_1'\overline{c_1}+d_2'\overline{c_2}=\overline{c_1}d_1'+\overline{c_2}d_2'$, $d_1'\overline{b}=d_1'b'$, $b'd_1'=\overline{b}d_1'$, $c_1'c_4+c_2'c_4=c_3c_1'+c_3c_2'$,$c_1'd_1+c_2'd_1=\overline{d}_1d_2'+\overline{d}_1d_1'$, $c_2'c_3+c_1'c_3=c_4c_1'+c_4c_2'$, $c_2'd_2+c_1'd_2=\overline{d}_2d_2'+\overline{d}_2d_1'$, $d_2'\overline{d}_1+d_1'\overline{d}_1=d_1c_1'+d_1c_2'$, $d_2'\overline{c_4}+d_1'\overline{c_4}=\overline{c_3}d_1'+\overline{c_3}d_2'$, $d_2'\overline{d}_2+d_1'\overline{d}_2=d_2c_1'+d_2c_2'$, $d_1'\overline{c_3}+d_2'\overline{c_3}=\overline{c_4}d_1'+\overline{c_4}d_2'$.

$\\$
\noindent {\bf Case 3, $(ABC,D)$} Resolving $p$ first, we shall get the following equations.

$(E_+,E_+)=-\{(bE_{-}+c_1E+c_3HC+d_2VT,E_+)+(c_2W+c_4HT+d_1VC,N_-)\}\\$
$-(bE_{-},E_+)=b\{(E_{-},b'E_{-}+c_1'E+d_1'S)+(W_{-},c_2'W+d_2'N)\}\\$
$-(c_1E,E_+)=c_1\{(E,b'E_{-}+c_1'E+d_1'S)+(HT,c_2'W+d_2'N)\}\\$
$-(c_3HC,E_+)=c_3\{(HC,b'E_{-}+c_1'E+d_1'S)+(W,c_2'W+d_2'N)\}\\$
$-(d_2VT,E_+)=d_2\{(VT,b'E_{-}+c_1'E+d_1'S)+(VC,c_2'W+d_2'N)\}\\$
$-(c_2W,N_-)=c_2\{(W,\overline{b}'N_{+}+\overline{c}_1'N+\overline{d}_2'W)+(HC,\overline{c}_2'S+\overline{d}_1'E)\}\\$
$-(c_4HT,N_-)=c_4\{(HT,\overline{b}'N_{+}+\overline{c}_1'N+\overline{d}_2'W)+(E,\overline{c}_2'S+\overline{d}_1'E)\}\\$
$-(d_1VC,N_-)=d_1\{(VC,\overline{b}'N_{+}+\overline{c}_1'N+\overline{d}_2'W)+(VT,\overline{c}_2'S+\overline{d}_1'E)\}\\$

\begin{table}[ht]
\caption{Case $(ABC,D)$, resolving $p$ first.}\label{tab:3a}
\begin{center}
\begin{tabular}[pos]{|l|l|l|l|l|l|l|l|l|l|l|l|}
\hline  p $\backslash $ q & E & $E_-$ & $W$ & $N$ & $N_+$ & S  \\
\hline $E$ & $c_4\overline{d}_1',c_1c_1'$ & $c_1b'$ &   &  &  &  $c_4\overline{c_2}',c_1d_1'$ \\
\hline $E_-$ & $bc_1'$ & $bb'$ &   &  &   &  $bd_1'$ \\
\hline $W$ &   &   & $c_2\overline{d}_2',c_3c_2'$ & $c_2\overline{c_1}',c_3d_2'$ & $c_2\overline{b}'$ &     \\
\hline $W_-$ &   &   & $bc_2'$ & $bd_2'$ &    &    \\
\hline $HC$ & $c_2\overline{d}_1',c_3c_1'$ & $c_3b'$ &   &   &    & $c_2\overline{c_2}',c_3d_1'$     \\
\hline $HT$ &   &  & $c_4\overline{d}_2',c_1c_2'$ & $c_4\overline{c_1}',c_1d_2'$ & $c_4\overline{b}'$ &    \\
\hline $VC$ &   &   & $d_1\overline{d}_2',d_2c_2'$ & $d_1\overline{c_1}',d_2d_2'$ & $d_1\overline{b}'$ &      \\
\hline $VT$ & $d_1\overline{d}_1',d_2c_1'$ & $d_2b'$ &   &  &  & $d_1\overline{c_2}',d_2d_1'$   \\
\hline
\end{tabular}
\end{center}
\end{table}

Resolving $q$ first, we shall get the following equations.

$(E_+,E_+)=-\{(E_+,b'E_{-}+c_1'E+d_1'S)+(W_+,c_2'W+d_2'N)\}\\$
$-(E_+,b'E_{-})=b'\{(bE_{-}+c_1E+c_3HC+d_2VT,E_{-})+(c_2W+c_4HT+d_1VC,N_{+})\}\\$
$-(E_+,c_1'E)=c_1'\{(bE_{-}+c_1E+c_3HC+d_2VT,E)+(c_2W+c_4HT+d_1VC,W)\}\\$
$-(E_+,d_1'S)=d_1'\{(bE_{-}+c_1E+c_3HC+d_2VT,S)+(c_2W+c_4HT+d_1VC,N)\}\\$
$-(W_+,c_2'W)=c_2'\{(bW_{-}+c_1W+c_4HT+d_1VC,W)+(c_2E+c_3HC+d_2VT,E)\}\\$
$-(W_+,d_2'N)=d_2'\{(bW_{-}+c_1W+c_4HT+d_1VC,N)+(c_2E+c_3HC+d_2VT,S)\}\\$

\begin{table}[ht]
\caption{Case $(ABC,D)$, resolving $q$ first.}\label{tab:3b}
\begin{center}
 \begin{tabular}[pos]{|l|l|l|l|l|l|l|l|l|l|l|l|}
\hline  p $\backslash $ q & E & $E_-$ & $W$ & $N$ & $N_+$ & S  \\
\hline $E$ & $c_2c_2',c_1c_1'$ & $c_1b'$ &   &  &  &  $c_2d_2',c_1d_1'$ \\
\hline $E_-$ & $bc_1'$ & $bb'$ &   &  &   &  $bd_1'$ \\
\hline $W$ &   &   & $c_1c_2',c_2c_1'$ & $c_1d_2',c_2d_1'$ & $c_2b'$ &     \\
\hline $W_-$ &   &   & $bc_2'$ & $bd_2'$ &    &    \\
\hline $HC$ & $c_3c_2',c_3c_1'$ & $c_3b'$ &   &   &    & $c_3d_2',c_3d_1'$     \\
\hline $HT$ &   &  & $c_4c_2',c_4c_1'$ & $c_4d_2',c_4d_1'$ & $c_4b'$ &    \\
\hline $VC$ &   &   & $d_1c_2',d_11'$ & $d_1d_2',d_1d_2'$ & $d_1b'$ &      \\
\hline $VT$ & $d_2c_2',d_2c_1'$ & $d_2b'$ &   &  &  & $d_2d_2',d_2d_1'$   \\
\hline
\end{tabular}
\end{center}
\end{table}

\noindent The relations here are:  $c_4\overline{d}_1'=c_2c_2'$, $c_4\overline{c_2}'=c_2d_2'$, $c_2\overline{d}_2'+c_3c_2'=c_1c_2'+c_2c_1'$, $c_2\overline{c_1}'+c_3d_2'=c_1d_2'+c_2d_1'$, $c_2\overline{b}'=c_2b'$, $c_2\overline{d}_1'=c_3c_2'$,  $c_2\overline{c_2}'=c_3d_2'$, $c_4\overline{d}_2'+c_1c_2'=c_4c_2'+c_4c_1'$, $c_4\overline{c_1}'+c_1d_2'=c_4d_2'+c_4d_1'$, $c_4\overline{b}'=c_4b'$, $d_1\overline{d}_2'+d_2c_2'=d_1c_2'+d_1c_1'$, $d_1\overline{c_1}'+d_2d_2'=d_1d_2'+d_1d_2'$, $d_1\overline{b}'=d_1b'$, $d_1\overline{d}_1'=d_2c_2'$, $d_1\overline{c_2}'=d_2d_2'$.

$\\$
\noindent {\bf Case 4, $(ACBD)$} Resolving $p$ first, we shall get the following equations.

\noindent \resizebox{12.5cm}{!} {$(E_+,E_+)=-\{(bE_{-}+c_1E,E_+)+(c_2W,W_+)+(c_4HT,S_-)+(c_3HC,N_-)+(d_2VT,N_-)+(d_1VC,S_-)\}\\$}

\noindent \resizebox{12.5cm}{!} {$-(bE_{-},E_+)=b\{(E_{-},bE_{-}+c_1E)+(W_-,c_2W)+(S_+,c_3HC)+(N_+,c_4HT)+(S_+,d_2VT)+(N_+,d_1VC)\}\\$}

\noindent $-(c_1E,E_+)=c_1\{(E,b'E_{-}+c_1'E)+(W,c_2'W)+(HC,d_2'N)+(HT,d_1'S)\}\\$
$-(c_2W,W_+)=c_2\{W,b'W_{-}+c_1'W)+(E,c_2'E)+(HC,d_2'N)+(HT,d_1'S)\}\\$
$-(c_4HT,S_-)=c_4\{(HT,\overline{b}'S_{+}+\overline{c}_1'S)+(HC,\overline{c}_2'N)+(E,\overline{d}_2'E)+(W,\overline{d}_1'W)\}\\$
$-(c_3HC,N_-)=c_3\{(HC,\overline{b}'N_{+}+\overline{c}_1'N)+(HT,\overline{c}_2'S)+(E,\overline{d}_1'E)+(W,\overline{d}_2'W)\}\\$
\noindent \resizebox{12.5cm}{!} {$-(d_2VT,N_-)=d_2\{(VT,\overline{b}N_{+}+\overline{c}_1N)+(VC,\overline{c}_2S)+(N,\overline{c}_3VC)+(S,\overline{c}_4VT)+(S,\overline{d}_1HC)+(N,\overline{d}_2HT)\}\\$}

\noindent \resizebox{12.5cm}{!} {$-(d_1VC,S_-)=d_1\{(VC,\overline{b}S_{+}+\overline{c}_1S)+(VT,\overline{c}_2N)+(N,\overline{c}_3VC)+(S,\overline{c}_4VT)+(S,\overline{d}_1HC)+(N,\overline{d}_2HT)\}\\$}

\begin{table}[ht]
\caption{Case $(ACBD)$, resolving $p$ first.}\label{tab:4a}
\begin{center}
\resizebox{12.5cm}{!}{
\begin{tabular}[pos]{|l|l|l|l|l|l|l|l|l|l|l|l|l|l|l|l|}
\hline  p $\backslash $ q & E & $E_-$ & W & $W_-$ & S & $S_+$ & N & $N_+$ & HC & HT & VC & VT  \\
\hline $E$ & $c_4\overline{d}_2',c_3\overline{d}_1'\atop c_2c_2',c_1c_1'$ & $c_1b'$ &   &  &    &   &   &  & &  & & \\
\hline $E_-$ & $bc_1$ & $bb$ &  &   &  &   &   &   &   &   &   &   \\
\hline $W$ &   &   & $c_4\overline{d}_1',c_3\overline{d}_2'\atop c_2c_1',c_1c_2'$  & $c_2b'$ &   &   &   &   &   &  &   &   \\
\hline $W_-$ &   &   & $bc_2$ &   &   &   &   &   &   &   &   &   \\
\hline $S$ &   &   &    &   &   &  &   &  & $d_2\overline{d}_1,d_1\overline{d}_1$ &   &   & $d_1\overline{c_4},d_2\overline{c_4}$ \\
\hline $S_+$ &   &  &  &   &  &   &   &   & $bc_3$ &   &   & $bd_2$ \\
\hline $N$ &  &  &  &  &   &   &  &   &   & $d_2\overline{d}_2,d_1\overline{d}_2$ & $d_1\overline{c_3},d_2\overline{c_3}$ &   \\
\hline $N_+$ &  &   &   &  &  &   &  &  &   & $bc_4$ & $bd_1$ &   \\
\hline $HC$ &   &  &   &   &  &  & $c_3\overline{c_1}',c_4\overline{c_2}'\atop c_2d_2',c_1d_2'$ & $c_3\overline{b}'$ &   &   &   &   \\
\hline $HT$ &  &  &   &   & $c_4\overline{c_1}',c_3\overline{c_2}'\atop c_2d_1',c_1d_1'$ & $c_4\overline{b}'$ &   &   &   &   &   &  \\
\hline $VC$ &  &   &    &   & $d_1\overline{c_1},d_2\overline{c_2}$ & $d_1\overline{b}$ &   &   &   &   &   &   \\
\hline $VT$ &  &   &    &   &   &  & $d_2\overline{c_1},d_1\overline{c_2}$ & $d_2\overline{b}$ &   &  &   &   \\
\hline
\end{tabular}}
\end{center}
\end{table}

Resolving $q$ first, we shall get the following equations.

\noindent \resizebox{12.5cm}{!} {$(E_+,E_+)=-\{(E_+,bE_{-}+c_1E)+(W_+,c_2W)+(S_-,c_3HC)+(N_-,c_4HT)+(S_-,d_2VT)+(N_-,d_1VC)\}\\$}

\noindent \resizebox{12.5cm}{!} {$-(E_+,bE_{-})=b\{(bE_{-}+c_1E,E_-)+(c_2W,W_-)+(c_4HT,S_+)+(c_3HC,N_+)+(d_2VT,N_{+})+(d_1VC,S_+)\}\\$}

\noindent $-(E_+,c_1E)=c_1\{(b'E_{-}+c_1'E,E)+(c_2'W,W)+(d_2'N,HT)+(d_1'S,HC)\}\\$
$-(W_+,c_2W)=c_2\{(b'W_{-}+c_1'W,W)+(c_2'E,E)+(d_1'N,HT)+(d_2'S,HC)\}\\$
$-(S_-,c_3HC)=c_3\{(\overline{b}'S_{+}+\overline{c}_1'S,HC)+(\overline{c}_2'N,HT)+(\overline{d}_2'E,E)+(\overline{d}_1'W,W)\}\\$
$-(N_-,c_4HT)=c_4\{(\overline{b}'N_{+}+\overline{c}_1'N,HT)+(\overline{c}_2'S,HC)+(\overline{d}_1'E,E)+(\overline{d}_2'W,W)\}\\$
\noindent \resizebox{12.5cm}{!} {$-(S_-,d_2VT)=d_2\{(\overline{b}S_{+}+\overline{c}_1S,VT)+(\overline{c}_2N,VC)+(\overline{c}_4VT,N)+(\overline{c}_3VC,S)+(\overline{d}_1HC,N)+(\overline{d}_2HT,S)\}\\$}

\noindent \resizebox{12.5cm}{!} {$-(N_-,d_1VC)=d_1\{(\overline{b}N_{+}+\overline{c}_1N,VC)+(\overline{c}_2S,VT)+(\overline{c}_4VT,N)+(\overline{c}_3VC,S)+(\overline{d}_1HC,N)+(\overline{d}_2HT,S)\}\\$}

\begin{table}[ht]
\caption{Case $(ACBD)$, resolving $q$ first.}\label{tab:4b}
\begin{center}
\resizebox{12.5cm}{!}{\begin{tabular}[pos]{|l|l|l|l|l|l|l|l|l|l|l|l|l|l|l|l|}
\hline  p $\backslash $ q & E & $E_-$ & W & $W_-$ & S & $S_+$ & N & $N_+$ & HC & HT & VC & VT  \\
\hline $E$ & $\overline{d}_2'c_3,\overline{d}_1'c_4 \atop c_2'c_2,c_1'c_1$ & $c_1b$ &   &  &    &   &   &  & &  & & \\
\hline $E_-$ & $b'c_1$ & $bb$  &  &   &  &   &   &   &   &   &   &   \\
\hline $W$ &   &   & $\overline{d}_1'c_3,\overline{d}_2'c_4 \atop c_2'c_1,c_1'c_2$  & $c_2b$ &   &   &  &   &   &   &  &   \\
\hline $W_-$ &  &  & $b'c_2$  &   &   &  &  &  &   &   &   &  \\
\hline $S$ &  &   &   &  &   &   &   &   & $\overline{c_1}'c_3,\overline{c_2}'c_4\atop d_2'c_2,d_1'c_1$ &   &   & $\overline{c_1}d_2,\overline{c_2}d_1$ \\
\hline $S_+$ &   &   &   &   &   &   &   &   & $\overline{b}'c_3$ &  &   & $\overline{b}d_2$ \\
\hline $N$ &  &   &   &  &  &   &   &   &   & $\overline{c_1}'c_4,\overline{c_2}'c_3\atop d_1'c_2,d_2'c_1$ & $\overline{c_1}d_1,\overline{c_2}d_2$ &   \\
\hline $N_+$ &  &   &   &  &  &   &   &   &   &  $\overline{b}'c_4$ & $\overline{b}d_1$ &  \\
\hline $HC$ &   &   &    &   &   &   & $\overline{d}_1d_2,\overline{d}_1d_1$ & $c_3b$ &   &   &   &   \\
\hline $HT$ &   &   &    &   & $\overline{d}_2d_2,\overline{d}_2d_1$ & $c_3b$ &   &   &   &   &   &  \\
\hline $VC$ &   &   &    &   & $\overline{c_3}d_1,\overline{c_3}d_2$ & $d_1b$ &   &   &   &   &   &  \\
\hline $VT$ &   &   &   &   &   &   & $\overline{c_4}d_1,\overline{c_4}d_2$ & $d_2b$ &   &   &   &  \\
\hline
\end{tabular}}
\end{center}
\end{table}

\noindent The relations here are: $c_4\overline{d}_2'+c_3\overline{d}_1'+ c_2c_2'+c_1c_1'=\overline{d}_2'c_3+\overline{d}_1'c_4 + c_2'c_2+c_1'c_1$, $c_1b'=c_1b$, $bc_1=b'c_1$, $c_4\overline{d}_1'+c_3\overline{d}_2'+ c_2c_1'+c_1c_2'=\overline{d}_1'c_3+\overline{d}_2'c_4 + c_2'c_1+c_1'c_2$, $c_2b'=c_2b$, $bc_2=b'c_2$, $d_2\overline{d}_1+d_1\overline{d}_1=\overline{c_1}'c_3+\overline{c_2}'c_4+ d_2'c_2+d_1'c_1$, $d_1\overline{c_4}+d_2\overline{c_4}=\overline{c_1}d_2+\overline{c_2}d_1$,
$bc_3=\overline{b}'c_3$, $bd_2=\overline{b}d_2$, $d_2\overline{d}_2+d_1\overline{d}_2=\overline{c_1}'c_4+\overline{c_2}'c_3+ d_1'c_2+d_2'c_1$, $d_1\overline{c_3}+d_2\overline{c_3}= \overline{c_1}d_1+\overline{c_2}d_2$, $bc_4= \overline{b}'c_4$, $bd_1=\overline{b}d_1$, $c_3\overline{c_1}'+c_4\overline{c_2}'+ c_2d_2'+c_1d_2'=\overline{d}_1d_2+\overline{d}_1d_1$, $c_4\overline{b}'=c_3b$, $d_1\overline{c_1}+d_2\overline{c_2}=\overline{c_3}d_1+\overline{c_3}d_2$, $d_1\overline{b}=d_1b$, $d_2\overline{c_1}+d_1\overline{c_2}=\overline{c_4}d_1+\overline{c_4}d_2$, $d_2\overline{b}=d_2b$.

$\\$
\noindent {\bf Case 5, $(ACDB)$} Resolving $p$ first, we shall get the following equations.

$(E_+,E_+)=-\{(bE_{-}+c_1E+c_4HT+d_1VC,E_+)+(c_2W+c_3HC+d_2VT,W_+)\}\\$
$-(bE_{-},E_+)=b\{(E_{-},bE_{-}+c_1E+c_3HC+d_2VT)+(W_{-},c_2W+c_4HT+d_1VC)\}\\$
$-(c_1E,E_+)=c_1\{(E,bE_{-}+c_1E+c_3HC+d_2VT)+(HC,c_2W+c_4HT+d_1VC)\}\\$
$-(c_4HT,E_+)=c_4\{(HT,bE_{-}+c_1E+c_3HC+d_2VT)+(W,c_2W+c_4HT+d_1VC)\}\\$
$-(d_1VC,E_+)=d_1\{(VC,bE_{-}+c_1E+c_3HC+d_2VT)+(VT,c_2W+c_4HT+d_1VC)\}\\$
$-(c_2W,W_+)=c_2\{(W,bW_{-}+c_1W+c_4HT+d_1VC)+(HT,c_2E+c_3HC+d_2VT)\}\\$
$-(c_3HC,W_+)=c_3\{(HC,bW_{-}+c_1W+c_4HT+d_1VC)+(E,c_2E+c_3HC+d_2VT)\}\\$
$-(d_2VT,W_+)=d_2\{(VT,bW_{-}+c_1W+c_4HT+d_1VC)+(VC,c_2E+c_3HC+d_2VT)\}$

\begin{table}[ht]
\caption{Case $(ACDB)$, resolving $p$ first.}\label{tab:5a}
\begin{center}
\begin{tabular}[pos]{|l|l|l|l|l|l|l|l|l|}
\hline  p $\backslash $ q & $E_-$ & E & $W_-$ & W & HC & HT & VC & VT \\
\hline $E_-$ & $bb$ & $bc_1$ &   &  & $bc_3$ &  &  & $bd_2$  \\
\hline E & $c_1b$ & $c_3c_2,c_1c_1$ &   &  &  $c_3c_3,c_1c_3$ &&& $c_3d_2,c_1d_2$    \\
\hline $W_-$ &  &   &   & $bc_2$ &  & $bc_4$ & $bd_1$ &    \\
\hline W &   &   & $c_2b$ & $c_2c_1,c_4c_2$ &  & $c_2c_4,c_4c_3$ & $c_2d_1,c_4d_1$ &    \\
\hline HC &   &   & $c_3b$ & $c_3c_1,c_1c_2$  &    & $c_3c_4,c_1c_4$ & $c_3d_1,c_1d_1$ &     \\
\hline HT & $c_4b$  & $c_2c_2,c_4c_1$ &   &  & $c_2c_3,c_4c_3$ &&& $c_2d_2,c_4d_2$  \\
\hline VC & $d_1b$ & $d_1c_1,d_2c_2$ &   &  & $d_1c_3,d_2c_3$ &&& $d_1d_2,d_2d_2$  \\
\hline VT &   &   & $d_2b$ & $d_2c_1,d_1c_2$ &  & $d_1c_4,d_2c_4$ & $d_1d_1,d_2d_1$ &     \\
 \hline
\end{tabular}
\end{center}
\end{table}
Resolving $q$ first, we shall get the following equations.

$(E_+,E_+)=-\{(E_+,bE_{-}+c_1E+c_3HC+d_2VT)+(W_+,c_2W+c_4HT+d_1VC)\} \\$
$-(E_+,bE_{-})=b\{(bE_{-}+c_1E+c_4HT+d_1VC,E_{-})+(c_2W+c_3HC+d_2VT,W_{-})\}\\$
$-(E_+,c_1E)=c_1\{(bE_{-}+c_1E+c_4HT+d_1VC,E)+(c_2W+c_3HC+d_2VT,HT)\}\\$
$-(E_+,c_3HC)=c_3\{(bE_{-}+c_1E+c_4HT+d_1VC,HC)+(c_2W+c_3HC+d_2VT,W)\}\\$
$-(E_+,d_2VT)=d_2\{(bE_{-}+c_1E+c_4HT+d_1VC,VT)+(c_2W+c_3HC+d_2VT,VC)\}\\$
$-(W_+,c_2W)=c_2\{(bW_{-}+c_1W+c_3HC+d_2VT,W)+(c_2E+c_4HT+d_1VC,HC)\}\\$
$-(W_+,c_4HT)=c_4\{(bW_{-}+c_1W+c_3HC+d_2VT,HT)+(c_2E+c_4HT+d_1VC,E)\}\\$
$-(W_+,d_1VC)=d_1\{(bW_{-}+c_1W+c_3HC+d_2VT,VC)+(c_2E+c_4HT+d_1VC,VT)\}.\\$

\begin{table}[ht]
\caption{Case $(ACDB)$, resolving $q$ first.}\label{tab:5b}
\begin{center}
\begin{tabular}[pos]{|l|l|l|l|l|l|l|l|l|}
\hline  p $\backslash $ q & $E_-$ & E & $W_-$ & W & HC & HT & VC & VT \\
\hline $E_-$ & $bb$ & $bc_1$ &   &  & $bc_3$ &  &  & $bd_2$  \\
\hline E & $c_1b$ & $c_4c_2,c_1c_1$ &   &  &  $c_2c_2,c_1c_3$ &&& $c_2d_1,c_1d_2$    \\
\hline $W_-$ &  &   &   & $bc_2$ &  & $bc_3$ & $bd_1$ &    \\
\hline W &   &   & $c_2b$ & $c_1c_2,c_2c_3$ &  & $c_1c_3,c_2c_1$ & $c_1d_1,c_2d_2$ &    \\
\hline HC &   &   & $c_3b$ & $c_3c_2,c_3c_3$  &    & $c_3c_3,c_3c_1$ & $c_3d_1,c_4d_2$ &     \\
\hline HT & $c_4b$  & $c_4c_3,c_4c_1$ &   &  & $c_4c_2,c_4c_3$ &&& $c_4d_2,c_4d_1$  \\
\hline VC & $d_1b$ & $d_1c_3,d_1c_1$ &   &  & $d_1c_2,d_1c_3$ &&& $d_1d_1,d_1d_2$  \\
\hline VT &   &   & $d_2b$ & $d_2c_2,d_2c_3$ &  & $d_2c_3,d_2c_1$ & $d_2d_1,d_2d_2$ &     \\
 \hline
\end{tabular}
\end{center}
\end{table}

\noindent The relations here are: $c_3c_2=c_4c_2$, $c_3c_3=c_2c_2$, $c_3d_2=c_2d_1$, $bc_3=bc_4$, $c_2c_1+c_4c_2=c_1c_2+c_2c_3$, $c_2c_4+c_4c_3=c_1c_3+c_2c_1$, $c_2d_1+c_4d_1=c_1d_1+c_2d_2$, $c_3c_1+c_1c_2=c_3c_2+c_3c_3$, $c_3c_4+c_1c_4=c_3c_4+c_3c_1$, $c_1d_1=c_4d_2$, $c_4c_3=c_2c_2$, $c_2c_3=c_4c_2$, $c_2d_2=c_4d_1$, $d_2c_2=d_1c_3$, $d_2c_3=d_1c_2$, $d_2d_2=d_1d_1$, $d_2c_1+d_1c_2=d_2c_2+d_2c_3$, $d_1c_4+d_2c_4=d_2c_3+d_2c_1$.

\bigskip In short, here are all the relations if the two crossings are all positive.

 \noindent {\bf Case 1:}  $c_2'\overline{d}_1'=\overline{d}_1'c_2'$, $c_2'\overline{c_2}'=\overline{d}'_1d_2'$, $b'd_2'=\overline{b}'d_2'$, $b'c_2'=\overline{b}'c_2'$, $d_2'\overline{b}'=d_2'b'$, $d_1'd_2'+d_2'\overline{c_1}'=d_2'd_1'+\overline{c_1}'d_2'$, $d_1'c_2'+d_2'\overline{d}_2'=\overline{c_1}'c_2'+d_2'c_1'$, $d_2'\overline{d}_1'=\overline{c_2}'c_2'$, $d_2'\overline{c_2}'=\overline{c_2}'d_2'$, $c_2'\overline{b}'=c_2'b'$, $c_2'\overline{c_1}'+c_1'd_2'=\overline{d}_2'd_2'+c_2'd_1'$, $c_2'\overline{d}_2'+c_1'c_2'=\overline{d}_2'c_2'+c_2'c_1'$.

 \noindent {\bf Case 2:}  $b'c_1'=bc_1'$, $c_2'c_1+c_1'c_2=c_1c_2'+c_2c_1'$, $c_2'c_2+c_1'c_1=c_2c_2'+c_1c_1'$, $b'c_2'=bc_2'$, $d_1'\overline{c_2}+d_2'\overline{c_1}=\overline{c_2}d_1'+\overline{c_1}d_2'$, $d_2'\overline{b}=d_2'b'$, $b'd_2'=\overline{b}d_2'$, $d_1'\overline{c_1}+d_2'\overline{c_2}=\overline{c_1}d_1'+\overline{c_2}d_2'$, $d_1'\overline{b}=d_1'b'$, $b'd_1'=\overline{b}d_1'$, $c_1'c_4+c_2'c_4=c_3c_1'+c_3c_2'$,$c_1'd_1+c_2'd_1=\overline{d}_1d_2'+\overline{d}_1d_1'$, $c_2'c_3+c_1'c_3=c_4c_1'+c_4c_2'$, $c_2'd_2+c_1'd_2=\overline{d}_2d_2'+\overline{d}_2d_1'$, $d_2'\overline{d}_1+d_1'\overline{d}_1=d_1c_1'+d_1c_2'$, $d_2'\overline{c_4}+d_1'\overline{c_4}=\overline{c_3}d_1'+\overline{c_3}d_2'$, $d_2'\overline{d}_2+d_1'\overline{d}_2=d_2c_1'+d_2c_2'$, $d_1'\overline{c_3}+d_2'\overline{c_3}=\overline{c_4}d_1'+\overline{c_4}d_2'$.

 \noindent {\bf Case 3:}  $c_4\overline{d}_1'=c_2c_2'$, $c_4\overline{c_2}'=c_2d_2'$, $c_2\overline{d}_2'+c_3c_2'=c_1c_2'+c_2c_1'$, $c_2\overline{c_1}'+c_3d_2'=c_1d_2'+c_2d_1'$, $c_2\overline{b}'=c_2b'$, $c_2\overline{d}_1'=c_3c_2'$,  $c_2\overline{c_2}'=c_3d_2'$, $c_4\overline{d}_2'+c_1c_2'=c_4c_2'+c_4c_1'$, $c_4\overline{c_1}'+c_1d_2'=c_4d_2'+c_4d_1'$, $c_4\overline{b}'=c_4b'$, $d_1\overline{d}_2'+d_2c_2'=d_1c_2'+d_1c_1'$, $d_1\overline{c_1}'+d_2d_2'=d_1d_2'+d_1d_2'$, $d_1\overline{b}'=d_1b'$, $d_1\overline{d}_1'=d_2c_2'$, $d_1\overline{c_2}'=d_2d_2'$.

 \noindent {\bf Case 4:}  $c_4\overline{d}_2'+c_3\overline{d}_1'+ c_2c_2'+c_1c_1'=\overline{d}_2'c_3+\overline{d}_1'c_4 + c_2'c_2+c_1'c_1$, $c_1b'=c_1b$, $bc_1=b'c_1$, $c_4\overline{d}_1'+c_3\overline{d}_2'+ c_2c_1'+c_1c_2'=\overline{d}_1'c_3+\overline{d}_2'c_4 + c_2'c_1+c_1'c_2$, $c_2b'=c_2b$, $bc_2=b'c_2$, $d_2\overline{d}_1+d_1\overline{d}_1=\overline{c_1}'c_3+\overline{c_2}'c_4+ d_2'c_2+d_1'c_1$, $d_1\overline{c_4}+d_2\overline{c_4}=\overline{c_1}d_2+\overline{c_2}d_1$,
$bc_3=\overline{b}'c_3$, $bd_2=\overline{b}d_2$, $d_2\overline{d}_2+d_1\overline{d}_2=\overline{c_1}'c_4+\overline{c_2}'c_3+ d_1'c_2+d_2'c_1$, $d_1\overline{c_3}+d_2\overline{c_3}= \overline{c_1}d_1+\overline{c_2}d_2$, $bc_4= \overline{b}'c_4$, $bd_1=\overline{b}d_1$, $c_3\overline{c_1}'+c_4\overline{c_2}'+ c_2d_2'+c_1d_2'=\overline{d}_1d_2+\overline{d}_1d_1$, $c_4\overline{b}'=c_3b$, $d_1\overline{c_1}+d_2\overline{c_2}=\overline{c_3}d_1+\overline{c_3}d_2$, $d_1\overline{b}=d_1b$, $d_2\overline{c_1}+d_1\overline{c_2}=\overline{c_4}d_1+\overline{c_4}d_2$, $d_2\overline{b}=d_2b$.

 \noindent {\bf Case 5:}  $c_3c_2=c_4c_2$, $c_3c_3=c_2c_2$, $c_3d_2=c_2d_1$, $bc_3=bc_4$, $c_2c_1+c_4c_2=c_1c_2+c_2c_3$, $c_2c_4+c_4c_3=c_1c_3+c_2c_1$, $c_2d_1+c_4d_1=c_1d_1+c_2d_2$, $c_3c_1+c_1c_2=c_3c_2+c_3c_3$, $c_3c_4+c_1c_4=c_3c_4+c_3c_1$, $c_1d_1=c_4d_2$, $c_4c_3=c_2c_2$, $c_2c_3=c_4c_2$, $c_2d_2=c_4d_1$, $d_2c_2=d_1c_3$, $d_2c_3=d_1c_2$, $d_2d_2=d_1d_1$, $d_2c_1+d_1c_2=d_2c_2+d_2c_3$, $d_1c_4+d_2c_4=d_2c_3+d_2c_1$.

\begin{remark}
We list here the nontrivial relations when the two crossings are all positive. The above relations then should be completed by $\overline{\ {\ }}$ and $\widehat{\ }$ operations. Please refer to the discussion in case 1. The collection of all nontrivial relations will be denoted by $R$.

If the variables satisfy the relations in $R$, then $f_{pq}=f_{qp}$.
\end{remark}

\section{Proof of the main theorem}
To define the invariant on any oriented link diagram $D$, we shall first assume/add some additional data.

   (1) Suppose each link component has an {\bf orientation}. This is already given.

   (2) {\bf Order} the link components by integers: 1,2, $\cdots$, m.

   (3) On each component $k_i$, pick a {\bf base point} $p_i$.

An oriented link diagram with order of link components and base points is called a {\bf marked diagram}. Now, we travel through component $k_1$ from $p_1$ along its orientation. When we finish $k_1$, we shall pass to $k_2$ starting from $p_2$, $\cdots$.

\begin{definition} A crossing point is called {\bf bad} if it is first passed over, otherwise, it is called {\bf good}. A link diagram contains only good crossings is called a monotone or ascending diagram.
\end{definition}

Given a monotone diagram, each link component $k_i$ can be regarded as a map $k_i: S^1 \to R^3= R^2 \times R$, and the $S^1$ can be divided into two arcs $\alpha_i \cup \beta_i$, such that, (1) the map $\beta_i \to R^2 \times R \to R^2$ is an immersion, i.e., its image is the monotone diagram. (2) different points in $\beta_i$ has different $z$ coordinates (the third coordinate in $R^2 \times R =R^3$), hence $\beta_i \to R^2 \times R \to R$ is monotonously increasing. (3) the image of $\alpha_i$ is vertical, i.e. its projection on $R^2$ is one single point, i.e., a base point. (4) any point in $k_i$ has smaller $z$ coordinate than the points in $k_{i+1}$. The set of maps $\{k_i\}$ is called a {\bf geometric realization} of a monotone diagram.

\begin{lemma} A monotone diagram corresponds to a trivial link.
\end{lemma}

\noindent We do not use this lemma explicitly in this paper. It will help the readers to understand why we define the value for monotone diagram to be $v_n$. The proof is easy. We leave it as an exercise.

\bigskip Now we are going to construct the link invariant for oriented link diagrams. For a given marked link diagram, we can define an ordered pair $(c,d)$ of integers, called its index. Here $c$ is the crossing number of the diagram, and $d$ is the number of bad points of the diagram. $(c,d)<(c',d')$ if $c<c'$, or $c=c'$ and $d<d'$. Let $S(c,d)$ denote the set of all marked link diagrams with indices $\leq (c,d)$. Note that $S(c,0)$ contains exactly the monotone diagrams with $c$ crossing points.

Now let's study the skein relations. Take $f(E_{+})+bf(E_{-})+c_1f(E)+c_2f(W)+c_3f(HC)+c_4f(HT)+d_1f(VC)+d_2f(VT)=0$ for example, each term has a link diagram corresponding to it. If the diagram $E_+$ is marked, as mentioned at beginning of section 2, all the other diagrams are canonical oriented. What's more, $E_{-}$ is canonically marked using the same order, base points as $E_+$. Suppose the marked link diagram $E_{+}$ has index $(c,d)$, then $E_{-}$ has index $(c,d+1)$ or $(c,d-1)$, and all other diagrams has crossing number $c-1$. As we will show later, the invariant actually does not depend on the order and base points of the link diagram. This tells us that we can construct the invariant and prove its properties use induction on the index pair $(c,d)$. For example, suppose that $E_{-}$ has smaller index $(c,d-1)$, and the invariant is already defined for any diagram with index $\leq (c,d)$, then $f(E_+)$ is uniquely determined by the skein relation. We shall use this as the definition of $f(E_+)$.

For any integer $n>0$, we introduce a variable $v_n$, and suppose that $(1+b+d_1+d_2)v_n+(c_1+c_2+c_3+c_4)v_{n+1}=0$ is hold for all $n$.

\begin{prop} If $d_1'=d_2'$, then there is a function $f$ defined for marked link diagrams, satisfies the following properties.

\noindent (1) The value for any marked link diagram is uniquely defined. For any trivial link diagram $D\in S(0,0)$ with $n$ components, $f(D)=v_n$.

\noindent (2) Resolving at any crossing point, the invariant satisfies the skein relations.

\noindent (3) It is invariant under base point changes.

\noindent (4) $f(D)$ is invariant under Reidemeister moves that never involve more than $c$ crossings.

\noindent (5) It is invariant under changing order of components.
\end{prop}

\begin{proof}

The construction and proofs are all using induction on the index pair $(c,d)$, where $c$ is the crossing number of the diagram, and $d$ is the number of bad points of the diagram. It is obvious that $0\leq d\leq c$.

$\\$
\noindent {\bf The initial Step.} For a diagram of index $(0,0)$, namely a monotone diagram with no crossing points, define its value to be $v_n$, where $n$ is the number of components of the link.

\noindent Then the statements (1)-(5) are satisfied for diagrams inside $S(0,0)$. There is nothing to prove in this case.

$\\$
\noindent {\bf The inductive Step.}
Now suppose the statements (1)-(5) are proved for link diagrams with crossings strictly less than $c$. This means that for any marked oriented link diagram with crossings $<c$, the value of the invariant is uniquely defined, independent of choice of base points and ordering of link components. Hence we can choose base points and ordering of link components arbitrarily to define the invariant.

$\\$
\noindent {\bf Proof of the statement (1)}:

If the diagram $D$ has index $(c,0)$, then it is a monotone diagram. We define $f(D)$ to be $v_n$, where $n$ means that the link has $n$ components.

Suppose that $f(D)$ is defined for diagrams of index $\leq (c,d)$, where $d\geq 0$. If the diagram $D$ has index $(c,d+1)$, then it has bad points. We resolve the diagram at its first bad point. Then, in the corresponding skein equation, all the other diagrams are of smaller indices than $(c,d)$. Hence $f$ is defined for those diagrams. So $f(D)$ is uniquely determined by the skein relation. We take this as the definition of invariant for $D$. We shall prove later that if we resolve at other crossing point we shall get the same result.

\begin{remark}
We can similarly define the invariant for marked diagrams on $S^2$. Given a marked link diagram $D$ on $R^2$, we can also regard it as a marked diagram on $S^2$. However, for a marked link diagram $D$ on $S^2$, we can have many marked diagrams on $R^2$, depending on where we pick the $\infty$ point. All those marked diagrams on $R^2$ have the same value of invariant using the definition above. As a consequence, when we later prove the Reidemeister moves invariance, we can actually allow more "generalized Reidemeister moves". For example, if an outermost monogon contains the $\infty$ point, we can use the Reidemeister move I to reduce it.
\end{remark}

\noindent {\bf Proof of the statement (2)}.

For a link diagram $D$, if $D$ has one bad point, since $f(D)$ is defined by the skein relation, it satisfies the statement (1).
If $D$ has at least 2 bad points, and one resolve at a bad point $q$. If $q$ is the first bad point, then by definition, the equation is satisfied.
If not, denote the first bad point by $p$. If we resolve at $p$, we get many diagrams $D_1,D_2,\cdots $ and a linear sum $f_p(D)=\sum \alpha_i f(D_i)$ for some $\alpha_i$. Then by definition $f(D)=f_p(D)$.

We resolve each $D_i$ at $q$, then we get the linear sum $f_{q}(D_i)$. Each diagram $D_i$ has strictly lower indices than $(c,b)$. If $D_i$ has crossing number $c-1$, then skein equation is proved for resolving at any point. If $D_i$ has crossing number $c$, then it has $b-1$ bad points, and $q$ is also a bad point of $D_i$. In all cases, by induction hypothesis, $f(D_i)=f_{q}(D_i)$. Hence $f(D)=f_p(D)=\sum \alpha_i f_{q}(D_i)$.

On the other hand, if we resolve $D$ at $q$ first, we get many diagrams $D_1',D_2',\cdots $, each has strictly lower indices than $(c,b)$. Hence the statements (0)-(4) are satisfied. We get a linear sum $f_q(D)=\sum \beta_i f(D_i')$. We resolve each $D_i'$ at $p$, then we get the linear sum  $f_p(D_i')$. By the argument before and our induction hypothesis, $f(D_i')=f_p(D_i')$. Hence $f_q(D)=\sum \beta_i f_p(D_i')$. On the other hand, the ring is designed such that $\sum \beta_i f_p(D_i')=\sum \alpha_i f_{q}(D_i)$! (This is the equation $f_{pq}=f_{qp}$.)
$$\xymatrix{f(D) \ar@{=} [d]^{definition}_{1st\ bad\ point} &    \\ f_p(D)=\sum \alpha_i f(D_i) \ar@{=} [d]_{induction}^{hypothesis} & f_q(D)=\sum \beta_i f(D_i') \ar@{=} [d]_{induction}^{hypothesis} \\ \sum \alpha_i f_{q}(D_i) \ar@{=} [r]^{f_{pq}=f_{qp}} & \sum \beta_i f_p(D_i') }$$

Therefor, $f(D)=f_p(D)=\sum \alpha_i f_{q}(D_i)=\sum \beta_i f_p(D_i')=f_q(D)$. That is, if we resolve at $q$, the skein equation is satisfied.

\begin{cor} If one resolve any point (not necessarily bad), the skein equation is satisfied.
\end{cor}

\begin{proof} If $q$ is a good point of $D$, we make a crossing change at $q$ get a new diagram $D'$, then $q$ is bad point of $D'$. The above proves that if we resolve $D'$ at $q$ the skein equation is satisfied. But this the same equation as $D$ resolving at $q$.

\end{proof}

\noindent This means that one can resolve at any crossing point to calculate the invariant, not necessarily the first bad point.

\bigskip
To prove  statement (3), we need the following lemmas.

\begin{lem}\label{lem:4} (\cite{LR} Lemma 15.1) Suppose that $p$ and $q$ are two arcs in $R^2$ meeting only at their end points $A$ and $B$, and let $R$ be the compact region bounded by $p\cup q$. Suppose that $t_1, t_2, \cdots , t_n$ are arcs in $R$, each meeting $p\cup q$ at just its end points, one in $p$ and one in $q$. Suppose that every $t_i \cap t_j$ is at most one point, that intersections of arcs are transverse and there are no triple points. The graph, with vertices all intersections of these arcs and edges comprising $p\cup q \cup (\cup_i t_i )$, separates $R$ into collection of $v$-gons. Then amongst these  $v$-gons there is a 3-gon with an edge in $p$ and there is a 3-gon with an edge in $q$.
\end{lem}

Using the above lemma, and a modification of {\cite{M}} Lemma 5.1, we can prove the following lemma for marked link diagrams.

\begin{lem}\label{lem:5} (1) Each marked monotone link diagram $D$ on $S^2$ with $< c$ crossings can be transformed to the unlink diagram in $S(0,0)$ using Reidemeister moves, and at each step, the resulting diagram has crossing number $<c$.

\noindent (2) Let $D$ be a marked monotone knot diagram $D$ with $c$ crossings. $b$ is the base point. Starting from $b$, $p$ is the first crossing point. Then after resolve $D$ at $p$, any diagram $D_i$ with $c-1$ crossings is a diagram of unknot or unlink.

\noindent (3) Let $D$ be a marked monotone link diagram $D$ with $c$ crossings. $b$ is one base point. Starting from $b$, $p$ is the first crossing point. If the two arcs passing through $p$ are from same link component, then after resolve $D$ at $p$, for any the diagrams $D_i$ with $c-1$ crossings is a diagram of unknot or unlink.
\end{lem}

\begin{proof}
(1) The proof is a modification of the proof in {\cite{M}}. In a link diagram $D$, a loop is a part of a component that starts and ends at the same crossing. A innermost loop is called simple if it has no selfintersections. Two arcs bound a bigon if they have no selfintersections, have common initial and final points and no other intersections. A bigon is called simple if it does not contain smaller bigons and loops inside.

By an easy innermost argument, one knows that if $D$ has crossing then $D$ has a simple bigon or a simple loop.

Case 1. If there is a simple bigon, suppose the two arcs $p,q$ forming the bigon are from a same link component $l$.  Since the bigon is simple, it satisfies the condition of lemma~\ref{lem:4}. Then there is a 3-gon with an edge in $p$ and there is a 3-gon with an edge in $q$. Then one of them dose not contain the base point of $l$. Since this is a monotone diagram, one can move the 3-gon outside of the bigon by a Reidemeister 3 move. Thus the bigon is simplified. When there are no arcs in the bigon, one can remove the bigon by a Reidemeister 2 move.

Case 2, if there is a simple bigon, and the two arcs $p,q$ forming the bigon are from link component $l,m$. If at most one base point $A$ or $B$ of $l,m$ lies in the bigon, then we can deal with it as in case 1. If both $A,B$ lie in the bigon, say $A$ is on arc $a$, $B$ is on arc $b$. $a$ divides the bigon into two parts, one part, say $X$, does not contain $B$. The boundary of $X$ contain three parts, $p'\subset p,q' \subset q, a$. We regard $p'\cup a=p'{}'$ as one arc. The it forms a bigon with $q'$. Applying lemma~\ref{lem:4}, one can use Reidemeister 3 moves to remove arcs inside this bigon since $B$ is outside of it. When there is no arc pass this new bigon, one can use Reidemeister 3 moves to move $q$ to remove this bigon. Now the bigon contains at most one base point $B$. We can simplify it as above.

For a loop, we can regard it as a degenerate bigon and treat it similarly.

(2) We can assume there is a small open disk $U$ containing $b$. $U$ contains only one crossing, $p$. There are two type of smoothings, horizontal ($HC,HT,E,W$) and vertical ($S,N,VC,VT$). See Fig. \ref{fig4}. When one starts at $B$ and travel along $D$, one passes $p,A_1,A_2,B_1,B_2$.

For a horizontal smoothing, $D$ has only one component and $D_i$ has has two components $D_i^1, D_i^2$. If we take $b_1,b_2$ as base points, then $D'$ is a monotone diagram, hence is a diagram of unlink.

For a vertical smoothing, $D,D'{}'$ both have only one component. Since the arc $A_1A_2$ has smaller $z$ coordinate than the arc $B_2B_1$. The two arcs are also monotone with respect to $z$ coordinate, hence they can contract to the boundary of the disk $U$ without obstruction. Hence $D'{}'$  is a diagram of unknot.

(3) Follows easily from (2).
\end{proof}

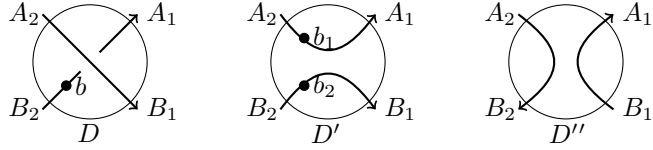
\begin{figure}[ht]
\beginpgfgraphicnamed{graphic-of-flat-world}
\begin{tikzpicture}[scale=.9]
\draw [thick] [->]  (0pt,40pt) -- (40pt,0pt); \draw [thick] (0pt,0pt) -- (16pt,16pt);
\draw [thick]  [->] (24pt,24pt) -- (40pt,40pt) ;
\draw [thick] [thick] (15pt,-10pt) node[text width=0.4pt] {$D$};
\draw [thick] [thick] (44pt,40pt) node[text width=0.4pt] {$A_1$};
\draw [thick] [thick] (-14pt,40pt) node[text width=0.4pt] {$A_2$};
\draw [thick] [thick] (44pt,0pt) node[text width=0.4pt] {$B_1$};
\draw [thick] [thick] (-14pt,0pt) node[text width=0.4pt] {$B_2$};
\draw [thick] [thick] (14pt,10pt) node[text width=0.4pt] {$b$};
\filldraw [black] (10pt,10pt) circle (2pt);
\draw  (20pt,20pt) ellipse (24pt and 24pt);

\draw [thick] [->] (100pt,0pt) .. controls (115pt,20pt) and (125pt,20pt) .. (140pt,0pt);
\draw [thick] (113pt,-10pt) node[text width=0.4pt] {$D'$};
\draw [thick] [->] (100pt,40pt) .. controls (115pt,20pt) and (125pt,20pt) .. (140pt,40pt);
\draw  (120pt,20pt) ellipse (24pt and 24pt);
\draw [thick] [thick] (144pt,40pt) node[text width=0.4pt] {$A_1$};
\draw [thick] [thick] (86pt,40pt) node[text width=0.4pt] {$A_2$};
\draw [thick] [thick] (144pt,0pt) node[text width=0.4pt] {$B_1$};
\draw [thick] [thick] (86pt,0pt) node[text width=0.4pt] {$B_2$};
\filldraw [black] (110pt,10pt) circle (2pt);
\draw [thick] [thick] (114pt,10pt) node[text width=0.4pt] {$b_2$};
\filldraw [black] (110pt,30pt) circle (2pt);
\draw [thick] [thick] (114pt,30pt) node[text width=0.4pt] {$b_1$};

\draw [thick]  [->] (200pt,40pt) .. controls (220pt,25pt) and (220pt,15pt) .. (200pt,0pt);
\draw [thick] (213pt,-10pt) node[text width=0.4pt] {$D'{}'$};
\draw [thick]  [<-] (240pt,40pt) .. controls (220pt,25pt) and (220pt,15pt) .. (240pt,0pt);
\draw  (220pt,20pt) ellipse (24pt and 24pt);
\draw [thick] [thick] (244pt,40pt) node[text width=0.4pt] {$A_1$};
\draw [thick] [thick] (186pt,40pt) node[text width=0.4pt] {$A_2$};
\draw [thick] [thick] (244pt,0pt) node[text width=0.4pt] {$B_1$};
\draw [thick] [thick] (186pt,0pt) node[text width=0.4pt] {$B_2$};
\end{tikzpicture}
\endpgfgraphicnamed
\caption{Resolve monotone diagram near base point.}
\label{fig4}
\end{figure}

\noindent {\bf Proof of the statement (3)}:

Given a diagram $D$ with a fixed orientation and order of components, suppose that there are two base point sets $B$ and $B'$. We only need to deal with the case that $B$ and $B'$ has only one point $x$ and $x'$ different, they are in the same component $k$, and between $x$ and $x'$ there is only one crossing point $p$. Using the base point sets $B$ or $B'$, $D$ has the same bad points except $p$. Let $f_B(D)$ and $f_{B'}(D)$ denote $f(D)$ using base point sets $B$ and $B'$ respectively.

We shall prove the equation $f_B(D)=f_{B'}(D)$. If there is bad point other than $p$, say $q$, we resolve $D$ at $q$ to get diagrams $D_1,D_2,\cdots $. Then those $D_i$'s has lower indices than $D$, hence base point invariance is proved for them. As before, we get a marked diagram $\overline{D}$ corresponding to crossing change at $q$. When we apply skein equation to the bad point, for $f_B(D)$ and $f_{B'}(D)$, each $f(D_i)$ has same value, hence  $f_B(D)=f_{B'}(D)$ if and only if  $f_B(\overline{D})=f_{B'}(\overline{D})$. Hence we can assume there are no other bad points.

Now there are three cases. Case 1. $p$ is a good point for both the two base point systems, then the values for $D$ are both $v_n$.
Case 2. $p$ is a bad point for both the two base point systems, then the skein equation tells the values are the same.

Case 3, $p$ is good in $B$, bad in $B'$. Then the two arcs passing through $p$ are from same link component, and the diagram $D$ with base point set $B$ is a monotone diagram. Applying the above lemma~\ref{lem:5}(3) to $D$, all $D_i$ are monotone diagrams. Suppose $D$ has $n$ components. Then $f_{B'}(D)$ is defined by the skein equation while all other terms $f_{B}(D)$ and $f(D_i)$ are known and have value in $\{ v_n,v_{n+1}\} $ .

On the other hand, $(1+b+d_1+d_2)v_n+(c_1+c_2+c_3+c_4)v_{n+1}=0$ is hold for all $n$. Since the two arcs passing through $p$ are from same link component, the $VC,VT$ diagrams all have $n$ components, the $HC,HT,W,E$ diagrams all have $n+1$ components. The values $f_{B}(D)$ and $f(D_i)$ fit the equations $(1+b+d_1+d_2)v_n+(c_1+c_2+c_3+c_4)v_{n+1}=0$.  Hence the solution of the skein equation is $f_{B'}=v_n$.

\bigskip
\noindent {\bf Proof of the statement (4)}:

\begin{lem} $f$ is invariant under Reidemeister III move.
\end{lem}
\begin{proof} Given two diagrams $D$ and $D'$, which differ by a Reidemeister move III.  Like above, we can assume all other points are good. In the two local disks containing the Reidemeister move III, there is a one to one correspondence between the three arcs as follows. We can order the three arcs by 1,2,3, ($1',2',3'$ in $D'$) such that arc 1 ($1'$) is above arc 2 ($2'$), and arc 2 ($2'$) is above arc 3 ($3'$). The one to one correspondence preserves the ordering. Their intersections induce a one to one correspondence between the three pairs of points in the two disks.  Call them $p,p'$, $q,q',r,r'$. If arc $i$ intersects arc $j$ at $x$, then arc $i'$ intersects arc $j'$ at $x'$.

Suppose $p$ is the intersection of arc 1 and arc 2 (or arc 2 and arc 3), then we resolve both $p$ and $p'$, and get many new link diagrams. There is a canonical one to one correspondence between those diagrams. So we can denote them by $D_1,D_2, \cdots , D_1',D_2', \cdots $. Here $D_1,D_1'$ correspond to crossing change for $D$ and $D'$, and all other diagrams are of smaller crossing numbers. By induction hypothesis, for those diagrams, we have $f(D_i)=f(D_i')$, $i\geq 2$. Therefor, by the skein equation, $f(D)=f(D')$ if and only if $f(D_1)=f(D_1')$. So we can assume $p$ is a good point. Similarly, we can assume the intersection of arc 2 and arc 3 is a good point.

Now, the intersection of arc 1 and arc 3, say $r$, is also a good point. The reason is simple. Since we proved base point invariance, we can assume there is no base point on any of the 3 arcs. The intersection of arc 2 and arc 3 is good means we first travel arc 3, then arc 2. Likewise, intersection of arc 1 and arc 2 is good means we first travel arc 2, then arc 1. Hence we first travel arc 3, then arc 1, the intersection of arc 1 and arc 3 is good.

So all the three intersections $p,q,r$ are good. It follows that $p',q',r'$ are good. Now we have two monotone diagrams, the invariance is clear.
\end{proof}

\begin{lem} $f$ is invariant under Reidemeister I move.
\end{lem}
\begin{proof}

\noindent Given two diagrams $D$ and $D'$, which differ by a Reidemeister one move. Say $D$ has index $(c-1,d)$, where $D'$ has index $(c,d')$. $D'$ has one extra crossing point $p$. By base point invariance, we can choose base point such that $p$ is a good point.

As before, we can assume that if there are bad points other than $p$, we can resolve them and prove Reidemeister move one invariance inductively.

Now, all other points are good, then $D$ and $D'$ are both monotone diagrams of trivial links. So $f(D)=f(D')$.

\end{proof}

\begin{lem} $f$ is invariant under Reidemeister II move.
\end{lem}
\begin{proof}

\noindent Given two diagrams $D$ and $D'$, which differs at a Reidemeister move II. $D'$ has two more crossings, $p$ and $q$. Likewise, we can assume all other points are good. If the two crossings, $p$ and $q$, one is good, the other is bad, one can use a base point change to make them both good. Then both the diagrams $D$ and $D'$ are monotone diagrams. There is nothing to prove.

The only case needs a proof is that both the two crossing are bad, and base point changes wouldn't change them from bad to good. However, changing both the two crossing will make them both good points. And in this case, the two arcs are from different link components.

In the following Fig. ~\ref{fig5}, we list all 3 possible cases of the intersections as $X_i,X_i'$, $i=1,2,3$. One has two bad points, the other has two good points. As before, we can assume all other points are good points. In each case, either $X_i$ or $X_i'$ is a monotone diagram. We apply the skein equation to the positive crossing of $X_i$ and $X_i'$ respectively. Then we have

$$f(X_i)+bf(Y_i)+c_1'f(E)+c_2'f(W)+d_1'f(S)+d_2'f(N)=0$$ and $$f(X_i')+bf(Y_i)+c_1'f(E')+c_2'f(W')+d_1'f(S')+d_2'f(N')=0.$$

\begin{figure}[ht]
\begin{tikzpicture}[scale=.25]
\draw[thick] [<-] (0,8) .. controls (4,4) and (4,4) .. (0,0); \draw (2,-1) node[text width=4pt] {$X_1$};
\draw[thick] [<-] (4,8) .. controls (2.5,6.7) and (2,6)  .. (2.5,6.5);
\draw[thick]  (1.5,5.5) .. controls (0.5,4) and (0.5,4)  .. (1.5,2.5);
\draw[thick]  (4,0) .. controls (3,0.9)   .. (2.3,1.6);
 \draw (4,2) node[text width=18pt] {$+$};
 \draw (4,6) node[text width=18pt] {$-$};

\draw[thick] [<-] (14,8) .. controls (10,4) and (10,4) .. (14,0); \draw (12,-1) node[text width=4pt] {$X_1'$};
\draw[thick] [<-] (10,8) .. controls (11.5,6.7) and (12,6)  .. (11.5,6.5);
\draw[thick]  (12.5,5.5) .. controls (13.5,4) and (13.5,4)  .. (12.5,2.5);
\draw[thick]  (10,0) .. controls (11,0.9)   .. (11.7,1.6);
\draw (14,2) node[text width=18pt] {$-$};
\draw (14,6) node[text width=18pt] {$+$};

\draw[thick] [<-] (20,8) .. controls (24,4) and (24,4) .. (22.5,2); \draw (22,-1) node[text width=4pt] {$Y_1$};
\draw[thick]  (20,0) .. controls (21,0.9)   .. (21.7,1.5);
\draw[thick] [<-] (24,8) .. controls (22.5,6.7) and (22,6)  .. (22.5,6.5);
\draw[thick]  (21.6,5.7) .. controls (20,3.6) and (20,4) .. (24,0);

\draw[thick] [shift={(0,-11)}] [<-] (0,8) .. controls (4,4) and (4,4) .. (0,0);
\draw  [shift={(0,-11)}] (2,-1) node[text width=4pt] {$X_2$};
\draw[thick] [shift={(0,-11)}]  (4,8) .. controls (2.5,6.7) and (2,6)  .. (2.5,6.5);
\draw[thick] [shift={(0,-11)}] (1.5,5.5) .. controls (0.5,4) and (0.5,4)  .. (1.5,2.5);
\draw[thick] [shift={(0,-11)}] [<-] (4,0) .. controls (3,0.9)   .. (2.3,1.6);
 \draw [shift={(0,-11)}] (4,2) node[text width=18pt] {$-$};
 \draw [shift={(0,-11)}] (4,6) node[text width=18pt] {$+$};

\draw[thick] [shift={(0,-11)}] [->] (14,8) .. controls (10,4) and (10,4) .. (14,0);
\draw  [shift={(0,-11)}] (12,-1) node[text width=4pt] {$X_2'$};
\draw[thick][shift={(0,-11)}] [<-] (10,8) .. controls (11.5,6.7) and (12,6)  .. (11.5,6.5);
\draw[thick] [shift={(0,-11)}] (12.5,5.5) .. controls (13.5,4) and (13.5,4)  .. (12.5,2.5);
\draw[thick] [shift={(0,-11)}] (10,0) .. controls (11,0.9)   .. (11.7,1.6);
\draw (14,2)[shift={(0,-11)}] node[text width=18pt] {$+$};
\draw (14,6) [shift={(0,-11)}]node[text width=18pt] {$-$};

\draw[thick] [shift={(0,-11)}]  (24,8) .. controls (20,4) and (20,4) .. (21.5,2.5);
\draw[thick] [shift={(0,-11)}] (22.8,5.5) .. controls (24,4) and (24,4) .. (20,0);
\draw[thick][shift={(0,-11)}] [<-] (20,8) .. controls (21.5,6.7) and (22,6)  .. (21.5,6.5);
\draw[thick] [shift={(0,-11)}] [<-] (24,0) .. controls (23,0.9)   .. (22.3,1.6);
\draw  [shift={(0,-11)}]  (22,-1) node[text width=4pt] {$Y_2$};

\draw[thick] [shift={(0,-22)}] [->] (0,8) .. controls (4,4) and (4,4) .. (0,0);
\draw  [shift={(0,-22)}] (2,-1) node[text width=4pt] {$X_3$};
\draw[thick] [shift={(0,-22)}] [<-] (4,8) .. controls (2.5,6.7) and (2,6)  .. (2.5,6.5);
\draw[thick] [shift={(0,-22)}] (1.5,5.5) .. controls (0.5,4) and (0.5,4)  .. (1.5,2.5);
\draw[thick] [shift={(0,-22)}]  (4,0) .. controls (3,0.9)   .. (2.3,1.6);
 \draw [shift={(0,-22)}] (4,2) node[text width=18pt] {$-$};
 \draw [shift={(0,-22)}] (4,6) node[text width=18pt] {$+$};

\draw[thick] [shift={(0,-22)}] [<-] (14,8) .. controls (10,4) and (10,4) .. (14,0);
\draw  [shift={(0,-22)}] (12,-1) node[text width=4pt] {$X_3'$};
\draw[thick][shift={(0,-22)}]  (10,8) .. controls (11.5,6.7) and (12,6)  .. (11.5,6.5);
\draw[thick] [shift={(0,-22)}] (12.5,5.5) .. controls (13.5,4) and (13.5,4)  .. (12.5,2.5);
\draw[thick] [shift={(0,-22)}] [<-] (10,0) .. controls (11,0.9)   .. (11.7,1.6);
\draw (14,2)[shift={(0,-22)}] node[text width=18pt] {$+$};
\draw (14,6) [shift={(0,-22)}]node[text width=18pt] {$-$};

\draw[thick] [shift={(0,-22)}] [<-]  (24,8) .. controls (20,4) and (20,4) .. (21.5,2.5);
\draw[thick] [shift={(0,-22)}] [->] (22.8,5.5) .. controls (24,4) and (24,4) .. (20,0);
\draw[thick][shift={(0,-22)}] (20,8) .. controls (21.5,6.7) and (22,6)  .. (21.5,6.5);
\draw[thick] [shift={(0,-22)}] (24,0) .. controls (23,0.9)   .. (22.3,1.6);
\draw [shift={(0,-22)}]  (22,-1) node[text width=4pt] {$Y_3$};

\end{tikzpicture}
%\endpgfgraphicnamed
\caption{Reidemeister move II invariance.}
\label{fig5}
\end{figure}
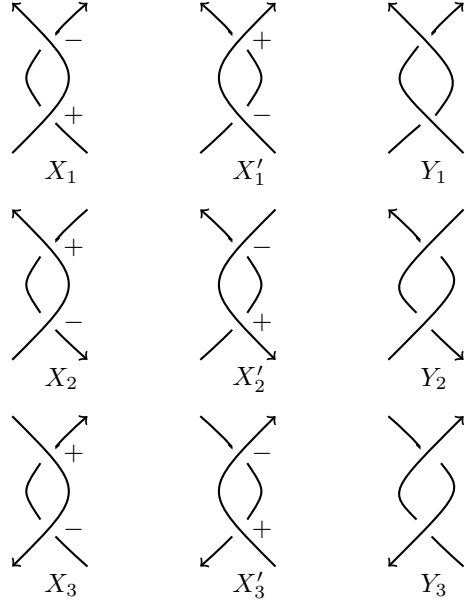

One can check that $f(E)=f(E'),f(W)=f(W'),f(S)=f(N'), f(N)=f(S')$. Since we assumed that $d_1'=d_2'$, we have $c_1'f(E)+c_2'f(W)+d_1'f(S)+d_2'f(N)=c_1'f(E')+c_2'f(W')+d_1'f(S')+d_2'f(N')$. Therefor, $f(X_i)=f(X_i')$ for $i=1,2,3$. Since either $X_i$ or $X_i'$ is a monotone diagram, invariance under Reidemeister II move is proved.

\end{proof}

\bigskip
\noindent {\bf Proof of the statement (5)}: $f$ is invariant under changing order of components.

\noindent Given two marked diagrams with different ordering of components. For simplicity, call them $D^1$ and $D^2$. By lemma ~\ref{lem:5}(1), they can be simultaneously reduced to trivial marked diagrams $\overline{D}^1$ and $\overline{D}^2\in S(0,0)$ by crossing changes and Reidemeister moves never increasing crossings. So $f(D^1)=f(D^2)$ if and only if $f(\overline{D}^1)=f(\overline{D}^2)$. However, $\overline{D}^1$ and $\overline{D}^2\in S(0,0)$ are trivial link diagrams with different ordering of link components.  By definition, $f(\overline{D}^1)=f(\overline{D}^2)=v_n$. Hence $f(D^1)=f(D^2)$.

\end{proof}

Now, let $X$ denote the quotient ring $Z[b,b',c_1,c_2,c_3,c_4,d_1,d_2,b',c_1'c_2',d_1',d_2',v_1,v_2,v_3,\cdots ]/R_1$, where $R_1=R\cup \{d_1'=d_2', (1+b+d_1+d_2)v_n+(c_1+c_2+c_3+c_4)v_{n+1}=0, $ for all $n=1,2,3, \cdots  \} $. Then we have the following theorem.

\begin{theorem}\label{sec:6}
For oriented link diagrams, there is a link invariant $f$ with values in $X$ and satisfies the following skein relations:

\noindent (1) If the two strands are from same link component, then

$f(E_{+})+bf(E_{-})+c_1f(E)+c_2f(W)+c_3f(HC)+c_4f(HT)+d_1f(VC)+d_2f(VT)=0.$

\noindent (2) Otherwise,
$f(E_{+})+b'f(E_{-})+c_1'f(E)+c_2'f(W)+d'_1f(S)+d'_2f(N)=0.$

The value for a trivial n-component link is $v_n$.

In general, replacing $X$ by any homomorphic image of $X$, one will get a link invariant.
\upshape
\end{theorem}
\medskip

\subsection{Modifying by writhe}\label{sec:2.3}

There is another closely related link invariant with values in another commutative ring $Y$. The idea is that the skein relations can reduce the calculation to monotone diagrams, and we can regard the set of monotone diagrams as a basis and assign writhe dependant values to those diagrams. The is an analogue of the Kauffman two variable polynomial.

Now, let $A$ be a new variable. Let $Y$ denote the quotient ring $\\ Z[A,A^{-1},b,b',c_1,c_2,c_3,c_4,d_1,d_2,b',c_1'c_2',d_1',d_2',v_1,v_2,v_3,\cdots ]/R_2$, where $R_2=R\cup \{d_1'=d_2',AA^{-1}=1,  Av_n+A^{-1}bv_n+(c_1+c_2+c_3+c_4)v_{n+1}+(d_1+d_2)v_n=0, $ for all $n=1,2,3, \cdots  \} $. Then we have the following theorem.

\begin{theorem}\label{sec:6}
There is a link invariant $F$ with values in $Y$. For oriented link diagram $D$, $F(D)=f(D)A^{-w}$ where $w$ is a the writhe of the link diagram, and $f$ satisfies the following skein relations.

\noindent (1) If the two strands are from same link component, then

$f(E_{+})+bf(E_{-})+c_1f(E)+c_2f(W)+c_3f(HC)+c_4f(HT)+d_1f(VC)+d_2f(VT)=0.$

\noindent (2) Otherwise,
$f(E_{+})+b'f(E_{-})+c_1'f(E)+c_2'f(W)+d'_1f(S)+d'_2f(N)=0.$

The value of $F$ for a trivial n-component link is $v_n$.

In general, replacing $Y$ by any homomorphic image of $Y$, one will get a link invariant.
\upshape
\end{theorem}

The following proposition gives the proof of the theorem.

\begin{prop} $f,F$ satisfy the following properties.

\noindent (1) For a monotone diagram $D$, $f(D)=A^wv_n$, $F(D)=v_n$, where $w$ is a the writhe of the link diagram, $n$ is the number of components.

\noindent (2) For any link diagram $D$, $F(D)=f(D)A^{-w}$.

\noindent (3) For any marked link diagram $D$, $f(D)$ and $F(D)$ are uniquely defined.

\noindent (4) The function $f$ satisfies type one skein relations if we resolve at any  point.

\noindent (5) The functions $f,F$ are invariant under base point change.

\noindent (6) $F$ is invariant under Reidemeister move I.

\noindent (7) $F$ and $f$ are invariant under Reidemeister moves II and III.

\noindent (8) $F$ and $f$ are invariant under changing order of components.

\end{prop}

\begin{proof}
As before, the proof is an induction on index $(c,d)$.

\bigskip\noindent {\bf Proof of the statement (1)(2)}:
There is nothing to prove.

\bigskip\noindent {\bf Proof of the statement (3)(4)}: As in last section, for a link diagram $D$ with bad point, we resolve it at the first bad point and use the skein equation to define $f(D)$. Then $f$ is defined inductively for all link diagrams. Since $f_{pq}=f_{qp}$ still hold, we have (4).

\bigskip\noindent {\bf Proof of the statement (5)}:

This is similar as in last section. Suppose that (5) is true for diagrams with crossings $\leq c$.

Given a diagram $D$ with a fixed orientation and order of components, suppose that there are two base point sets $B$ and $B'$. We only need to deal with the case that $B$ and $B'$ has only one point $x$ and $x'$ different, they are in the same component $k$, and between $x$ and $x'$ there is only one crossing point $p$. Using the base point sets $B$ or $B'$, $D$ has the same bad points except $p$. Let $f_B(D)$ and $f_{B'}(D)$ denote $f(D)$ using base point sets $B$ and $B'$ respectively.

As in last section, we can assume the crossings different from $p$ are all good points.  $p$ is good with respect to $B$, bad with respect to $B'$. Then the two arcs passing through $p$ are from same link component, and the diagram $D$ with base point set $B$ is a monotone diagram.

Suppose $D$ has $n$ components. Let $w(p)$ denote the sign of the crossing $p$, let $w$ denote the sum of signs of all other crossings of $D$. Let $w(D)$ denote the writhe of $D$. Then $w(D)=w(p)+w$.

\bigskip\noindent {\bf Claim}: $w(E)=w(W)=w(HC)=W(HT)=w(VC)=w(VT)=w$.

\noindent {\bf Proof of claim}: For a horizontal smoothing, one get two new link components by smoothing at $p$. See Fig. \ref{fig4}. For the horizontal smoothing, there is a choice of base points, orientation and order such that the result is a monotone diagram. Hence we can move each components such that the result is a disjoint union of knots diagrams. Denote it by $D_i^*$ Furthermore, we can assure that the diagram of each knot is not changed during the process. Hence it can be realized by a sequence of Reidemeister II and III moves. Notice that Reidemeister move II and III do not change writhe. Hence $w(D_i^*)=w(D_i)$.

On the other hand, for a knot, its writhe is independent of the choice of orientation. Hence for $D_i^*$, $w(D_i^*)$ is the same for all $E,W,HC,HT$. Hence $w(E)=w(W)=w(HC)=W(HT)$.

What's more, it is clear that $W(HT)=w(VC)$ and $w(HC)=w(VT)$.

Notice that $w(HC)=w$. So the claim is proved.

\bigskip

By induction hypothesis, $F$ is a knot invariant defined for diagrams with crossing $\leq c-1$.

If we resolve $D$ at $p$, we get diagrams $D_0,E,W,HC,HT,VC,VT$. (We also refer to them as $D_i$, $i=0,1,\cdots ,6$.) Applying the above lemma~\ref{lem:5}(3) to $D$, all diagrams $D_i$ are unlinks. Then $F(D_i)=v_{n+1}$ for horizontal smoothings and $F(D_j')=v_{n}$ for vertical smoothings. Also $F(D_0)=v_n$,
So $f(D_i)=A^wv_{n+1}$ for horizontal smoothings and $F(D_j')=A^wv_{n}$ for vertical smoothings.

\noindent Case 1.  If $p$ is a positive crossing, and if $F_B'(D)=v_n$, then $$f_{B'}(D)+bf_{B'}(D_0)=A^w\times (Av_n+A^{-1}bv_n).$$

\noindent Case 2.  If $p$ is a negative crossing, and if $F_B'(D)=v_n$, then $$f_{B'}(D_0)+bf_{B'}(D)=A^w\times (Av_n+A^{-1}bv_n).$$

On the other hand, $A^w\times (Av_n+A^{-1}bv_n+(c_1+c_2+c_3+c_4)v_{n+1}+(d_1+d_2)v_n) =0$ is hold for all $n$. $f_{B'}(D)$ is defined by the skein equation while all other terms are known. The unique solution in each case is $F_B'(D)=v_n$. Hence $F$ and $f$ are invariant under change of base points for diagrams having $c$ crossings.

\bigskip\noindent {\bf Proof of the statement (6)}: As in last section, we can move the base points such that crossing point in the Reidemeister move I is a good point. The proof is the same as in last section.

\bigskip\noindent {\bf Proof of the statement (7)}: Notice that Reidemeister move II does not change writhe. Check the following equations.

$$f(X_i)+bf(Y_i)+c_1'f(E)+c_2'f(W)+d_1'f(S)+d_2'f(N)=0$$ and $$f(X_i')+bf(Y_i)+c_1'f(E')+c_2'f(W')+d_1'f(S')+d_2'f(N')=0.$$

After considering writhe, we still have $f(E)=f(E'),f(W)=f(W'),f(S)=f(N'), f(N)=f(S')$. Since we assumed that $d_1'=d_2'$, we have $c_1'f(E)+c_2'f(W)+d_1'f(S)+d_2'f(N)=c_1'f(E')+c_2'f(W')+d_1'f(S')+d_2'f(N')$. Therefor, $f(X_i)=f(X_i')$ for $i=1,2,3$. Since either $X_i$ or $X_i'$ is a monotone diagram, invariance under Reidemeister II move is proved.

\bigskip\noindent {\bf Proof of the statement (8)(9)}: Notice that Reidemeister move III does not change writhe. The proofs are the same.

\end{proof}

If we let $A=1$, we see the first invariant is a special case of the modified invariant.

\section{Some simplification}\label{sec:6}

The modified invariant is a generalization of both HOMFLY polynomial and Kauffman two-variable polynomial. However, it is very complicated, hard to compute. There are some symmetric simplification of it. For example, we let $b=b'=b^{-1}$, then $x=\overline{x}$. Let $c/4=c_1=c_2=c_3=c_4,d/2=d_1=d_2,c'/2=c_1'=c_2',d'/2=d_1'=d_2'$, then $x=\widehat{x}$. The the relations are dramatically simplified.

Furthermore, we add the new relations $d'=bc', cv_{n+1}+(A+A^{-1}b+d)v_n=0$. Plug all those into the relation set $R$, the new relation set contains the following.

$$d'=bc',dc'=bc'c',b^2=1,dd=cc',cv_{n+1}=-(A+A^{-1}b+d)v_n, i\geq 1.$$

To go one step further, we need the famous diamond lemma {\cite{N}}. One can also consult Wikipedia.

\bigskip
\noindent {\bf An Equivalent version of the diamond lemma {\cite{N}}}:
For every binary relation with no decreasing infinite chains and satisfying the diamond property, there is a unique minimal element in every connected component of the relation considered as a graph.

\bigskip
Since $d'=bc'$, we delete the variable $d'$, use the following variables $A,b,c,c',d,v_n,n=1,2,\cdots $. Regard the above relations as a rewriting system as follows.
\begin{align}
dc'\rightarrow bc'c',b^2\rightarrow 1,dd\rightarrow cc',AA^{-1}\rightarrow 1, cv_{n+1}\rightarrow -(A+A^{-1}b+d)v_n, i\geq 1. \label{eq:rule}
\end{align}

Let $\deg v_{n+1}=8n, \deg d=4, \deg c=\deg c'=2, \deg b=1, \deg A= \deg A^{-1}=1$. Then one can see that the rewriting system always decreases the degree, hence there does not exist decreasing infinite chains.

To verify the diamond property, notice that for the value of $F(D)$, every term contains exactly one variable from $\{v_n, n=1,2,\cdots \}$, hence there is only the following one case to check.
$$d(dc')\rightarrow b(dc')c'\rightarrow (bb)c'c'c'\rightarrow c'c'c',$$ or $$(dd)c'\rightarrow cc'c'.$$ Hence we add the new relation $c'c'c'=cc'c'$. Now the rewriting system is as follows.

\begin{align}
dc'\rightarrow bc'c',b^2\rightarrow 1,dd\rightarrow cc',,AA^{-1}\rightarrow 1, c'c'c'\rightarrow cc'c', cv_{n+1}\rightarrow -(A+A^{-1}b+d)v_n, i\geq 1.
\label{eq:rule1}
\end{align}

Now, verify the diamond property again, there are two new cases to check.

$$(dc')c'c'\rightarrow b(c'c'c')c'\rightarrow bc(c'c'c')\rightarrow bccc'c',$$

or $$d(c'c'c')\rightarrow c(dc')c'\rightarrow cb(c'c'c')\rightarrow bccc'c'.$$

Here is another case.
$$d(dc'c'c')\rightarrow dbccc'c' = bcc(dc'c')  \rightarrow bcc b(c'c'c') \rightarrow bbcccc'c'\rightarrow cccc'c',$$ or
$$(dd)c'c'c'\rightarrow cc'c'c'c' \rightarrow ccc'c'c' \rightarrow cccc'c'.$$

Hence this rewriting system satisfies the condition of the diamond lemma, any result $F(D)$ (or $f(D)$) has a unique normal form.

Now, let $Z$ denote the quotient ring $Z[A,A^{-1},b,c,c',d,v_1,v_2,v_3,\cdots ]/R_3$, where $R_3=\{dc'=bc'c',b^2=1,dd=cc',AA^{-1}=1, cv_{n+1}+(A+A^{-1}b+d)v_n=0, i\geq 1 \} $. Then we have the following theorem.

\begin{theorem}\label{sec:6}
There is a link invariant $F$ with values in $Z$. For oriented link diagram $D$, $F(D)=f(D)A^{-w}$ where $w$ is a the writhe of the link diagram, and $f$ satisfies the following skein relations.

\noindent (1) If the two strands are from same link component, then

$f(E_{+})+bf(E_{-})+c(f(E)+f(W)+f(HC)+f(HT))/4+d(f(VC)+f(VT))/2=0.$

\noindent (2) Otherwise,
$f(E_{+})+bf(E_{-})+c'(f(E)+f(W))/2+bc'(f(S)+f(N))/2=0.$

The value of $F$ for a trivial n-component link is $v_n$.

And the rewriting rules (\ref{eq:rule1}) given a unique nomal form for the invariant $F$.
\upshape
\end{theorem}

\bigskip\noindent {\bf Example: The right hand trefoil.}

Let $H$ denote the minimal diagram of Hopf link, $H^*$ denote its mirror image, $D$ denote a minimal diagram of the right hand trefoil. Apply the skein equation to any crossing point then we get
$$f(H)=-bv_2-Ac'v_1-A^{-1}bc'v_1=-bv_2-c'(A+bA^{-1})v_1,$$
$$f(H^*)=-bv_2-bA^{-1}c'v_1-bAbc'v_1=-bv_2-c'(A+bA^{-1})v_1.$$

Then
\begin{flalign*}
f(D) &=-bAv_1-\frac{c}{2}(f(H)+f(H^*))-dA^{-2}v_1   \\
     &=-bAv_1+cbv_2 +cc'(A+bA^{-1})v_1 -dA^{-2}v_1  \\
     &=-bAv_1-b(A+A^{-1}b+d)v_1 +cc'(A+bA^{-1})v_1 -dA^{-2}v_1  \\
     &=((cc'-2b)A-bd+(bcc'-1)A^{-1}-dA^{-2})v_1  \\
\end{flalign*}

The invariant $F(D)=A^{-3}f(D)$.

In the above calculation we use the rewriting rules (\ref{eq:rule1}), and we use the $=$ sign instead of $\rightarrow$, since they are equal in the ring. The rewriting rules (\ref{eq:rule1}) gives the unique normal form. In the final result of any $F(D)$, the highest powers of $d,b$ are $\leq 1$, the highest power of $c'$ is $\leq 2$. This is a generalization of the 2-variable Kauffman polynomial.

\section*{Acknowledgements}
The author would like to thank Ruifeng Qiu, Jiajun Wang, Ying Zhang, Xuezhi Zhao, Hao Zheng, Teruhisa Kadokami for helpful discussions. The author also thanks all the organizers of 2009 summer school of knot theory in ICTP, Italy. He really had a great time there and got many inspirations. The author also thanks NSF of China for providing traveling fee for the trip, and all the efforts made by ICTP, Italy.


\nocite{*}

\begin{thebibliography}{Ag}

\bibitem{A}  C.C. Adams, {\it The Knot Book}, W.H. Freeman and Company (1999)

\bibitem {Al}  J.W. Alexander, {\it Topological Invariants of Knots and Links}, Transactions of the American Mathematical Society, Volume 30, Issue 2 (April 1928), 275-306.

\bibitem{GH} G. Burde, and H. Zieschang, {\it Knots}, de Gruyter Studies in Mathematics, 5, Walter de Gruyter, Berlin (1985).

\bibitem{HOMFLY} P. Freyd, D. Yetter, J. Hoste, W. B. R. Lickorish, K. Millett, and A. Ocneanu, {\it A new polynomial invariant of knots and links}, Bull. Amer. Math. Soc. 12 (1985), 239-249.

\bibitem{J}  V. F. R. Jones, {\it Hecke algebra representations of braid groups and link polynomials}, Ann. of Math. (2) 126 (1987), 335-388.

\bibitem{K}  L. Kauffman, {\it State models and the Jones polynomial}, Topology 26 (1987), 395-407.

\bibitem{M}  V. O. Manturov, {\it Knot Theory}. CRC Press, 2004.

\bibitem{N}  M. H. A. Newman. {\it On theories with a combinatorial definition of "equivalence"}, Ann. of Math. (2) 43 (1942), 223¨C243.

\bibitem{PT} J. H. Przytycki and P. Traczyk, {\it Invariants of links of Conway type}, Kobe J. Math. 4 (1987), 115-139.

\bibitem{R} D. Rolfsen, {\it Knots and links}, Publish or Perish Inc., Berkeley, (1976), 160-197.

\bibitem{LR} W.B.R. Lickorish, {\it An introduction to knot theory}. Springer-Verlag, New York, 1997.

\bibitem{Y} Z. Yang, {\it New link invariants and Polynomials (I), oriented case}. arXiv:1004.2085.

\end{thebibliography}
\end{document}